\documentclass[11pt, final]{article}
\usepackage{a4}
\usepackage{amsmath}%
\usepackage{amstext}%
\usepackage{amssymb}%
\usepackage{showkeys}%
\usepackage{epsfig}%
\usepackage{cite}
\usepackage{tikz}
\usepackage{subfig}
\usepackage{caption}
\usepackage{multirow}
\usepackage{booktabs}
\usepackage{floatrow}

\usepackage{listings}

\newtheorem{theorem}{Theorem}

\newtheorem{axiom}{Axiom}

\newtheorem{corollary}[theorem]{Corollary}

\newtheorem{definition}[axiom]{Definition}

\newtheorem{lemma}[theorem]{Lemma}

\newtheorem{proposition}[theorem]{Proposition}
\newtheorem{rem}[theorem]{Remark}
\newenvironment{remark}{\rem\rm}{\endrem}

\newcounter{unnumber}

\newtheorem{prf}[unnumber]{Proof.}
\newenvironment{proof}{\prf\rm}{\hfill{$\blacksquare$}\endprf}
\newcommand{\R}{\mathbb{R}}%
\newcommand{\N}{\mathbb{N}}%
\newcommand{\e}{\varepsilon}%
\newcommand{\ol}{\overline}%
\newcommand{\ul}{\underline}%

\newcommand{\n}{{\nabla}}
\newcommand{\p}{{\partial}}
\newcommand{\ds}{\displaystyle}

\def\a{\alpha}
\def\oa{\overline{\alpha}}
\def\ua{\underline{\alpha}}
\def\b{\beta}
\def\e{\epsilon}

\def\m{\mu}
\def\s{\sigma}

\def\<{\langle}
\def\>{\rangle}
\DeclareMathOperator*\dom{dom}%
\DeclareMathOperator*\prox{prox}%
\DeclareMathOperator*\argmin{argmin}

\DeclareMathOperator*\crit{crit}
\DeclareMathOperator*\dist{dist}
\DeclareMathOperator*\sgn{sgn}

\title{An inertial forward-backward algorithm for the minimization of the sum of two nonconvex functions}

\author{Radu Ioan Bo\c{t} \thanks{University of Vienna, Faculty of Mathematics, Oskar-Morgenstern-Platz 1, A-1090 Vienna, Austria, 
email: radu.bot@univie.ac.at. Research partially supported by DFG (German Research Foundation), project BO 2516/4-1.} \and 
Ern\"{o} Robert Csetnek \thanks {University of Vienna, Faculty of Mathematics, Oskar-Morgenstern-Platz 1, A-1090 Vienna, Austria, 
email: ernoe.robert.csetnek@univie.ac.at. Research supported by DFG (German Research Foundation), project BO 2516/4-1.} 
 \and Szil\'{a}rd Csaba L\'{a}szl\'{o} \thanks{Technical University of Cluj-Napoca, Department of Mathematics, 400027 Cluj-Napoca, Romania, e-mail: 
  szilard.laszlo@math.utcluj.ro}}

\begin{document}
\maketitle

\noindent \textbf{Abstract.} We propose a forward-backward proximal-type algorithm with inertial/memory effects for  
minimizing the sum of a nonsmooth function with 
a smooth one in the nonconvex setting. The sequence of iterates generated by the algorithm converges to a critical point of 
the objective function provided an appropriate regularization of the objective satisfies the 
Kurdyka-\L{}ojasiewicz inequality, which is for instance fulfilled for semi-algebraic functions. We illustrate 
the theoretical results by considering two numerical experiments: the first one concerns the ability of recovering 
the local optimal solutions of nonconvex optimization problems, while the second one refers to the restoration of a noisy blurred image. \vspace{1ex}

\vspace{1ex}


\section{Introduction}\label{sec-intr}

Proximal-gradient splitting methods are powerful techniques used in order to solve optimization problems 
where the objective to be minimized is the sum of a finite collection of smooth and/or nonsmooth functions. The main feature 
of this class of algorithmic schemes is the fact that they access each function separately, either by a gradient step 
if this is smooth or by a proximal step if it is nonsmooth. 

In the convex case (when all the functions involved are convex), these methods are well understood, see for example 
\cite{bauschke-book}, where the reader can find a presentation of the most prominent methods, like the forward-backward, forward-backward-forward and 
the Douglas-Rachford splitting algorithms. 

On the other hand, the nonconvex case is less understood, one of the main difficulties coming from the fact that the proximal point operator is in general not anymore single-valued. However, 
one can observe a considerably progress in this direction when the functions in the objective have the \textit{Kurdyka-\L{}ojasiewicz property} (so-called \textit{KL functions}), as it is the case for the ones
with different analytic features.  This applies for both the forward-backward algorithm (see \cite{b-sab-teb}, \cite{att-b-sv2013}) and the forward-backward-forward algorithm (see \cite{b-c-inertial-nonc-ts}). We refer 
the reader also to \cite{attouch-bolte2009, att-b-red-soub2010, c-pesquet-r, f-g-peyp, h-l-s-t, ipiano} for literature concerning 
proximal-gradient splitting methods in the nonconvex case relying on the \textit{Kurdyka-\L{}ojasiewicz property}. 

A particular class of the proximal-gradient splitting methods are the ones with inertial/memory effects. These iterative schemes have their origins in the time discretization of some differential 
inclusions of second order type (see \cite{alvarez2000, alvarez-attouch2001}) and share the feature that the new iterate is defined by using the previous two iterates. The increasing interest in this class of algorithms is emphasized by a considerable number of papers written in the last fifteen years on this topic,  see  
\cite{alvarez2000, alvarez-attouch2001, alvarez2004, att-peyp-red, b-c-inertial, b-c-inertial-admm, b-c-inertialhybrid, b-c-inertial-nonc-ts, b-c-h-inertial, mainge2008, mainge-moudafi2008, moudafi-oliny2003, cabot-frankel2011, pesq-pust, chen-ma-yang, chan-ma-yang}. 

Recently, an inertial forward-backward type algorithm has been proposed and analyzed in \cite{ipiano} in the nonconvex setting, by assuming that the nonsmooth part of the objective function is convex, 
while the smooth counterpart is allowed to be nonconvex.  It is the aim of this paper to introduce an inertial forward-backward algorithm in the full nonconvex setting and to study its convergence properties. 
The techniques for proving the convergence of the numerical scheme use the same three main ingredients, as other algorithms for nonconvex optimization problems involving 
KL functions. More precisely, we show a sufficient decrease property for the iterates, the existence of a subgradient lower bound for the iterates gap and, finally, we use the analytic features of the objective function in order 
to obtain convergence, see \cite{b-sab-teb, att-b-sv2013}. The  {\it limiting (Mordukhovich) subdifferential} and its properties play an important role in the analysis. The main result of this paper shows that,
provided an appropriate regularization of the objective satisfies the Kurdyka-\L{}ojasiewicz property, the convergence of the inertial forward-backward algorithm is guaranteed. As a particular instance, 
we also treat the case when the objective function is semi-algebraic and present the convergence properties of the algorithm. 

In the last section of the paper we consider two numerical experiments. The first one has an academic character and shows the ability of
algorithms with inertial/memory effects to detect optimal solutions which are not found by the non-inertial versions (similar allegations can be found also in \cite[Section 5.1]{ipiano} and 
\cite[Example 1.3.9]{bertsekas}). The second one concerns the restoration of a noisy blurred image by using a nonconvex misfit functional with nonconvex regularization.

\section{Preliminaries}

In this section we recall some notions and results which are needed throughout this paper. Let $\N= \{0,1,2,...\}$ be the set of nonnegative integers. For $m\geq 1$, the Euclidean scalar product and the induced norm on $\R^m$
are denoted by $\langle\cdot,\cdot\rangle$ and $\|\cdot\|$, respectively. Notice that all the finite-dimensional spaces considered in the 
manuscript are endowed with the topology induced by the Euclidean norm. 

The {\it domain} of the function  $f:\R^m\rightarrow (-\infty,+\infty]$ is defined by $\dom f=\{x\in\R^m:f(x)<+\infty\}$. We say that $f$ is {\it proper} if $\dom f\neq\emptyset$.  
For the following generalized subdifferential notions and their basic properties we refer to \cite{boris-carte, rock-wets}. 
Let $f:\R^m\rightarrow (-\infty,+\infty]$ be a proper and lower semicontinuous function. If $x\in\dom f$, we consider the {\it Fr\'{e}chet (viscosity)  
subdifferential} of $f$ at $x$ as the set $$\hat{\partial}f(x)= \left \{v\in\R^m: \liminf_{y\rightarrow x}\frac{f(y)-f(x)-\<v,y-x\>}{\|y-x\|}\geq 0 \right \}.$$ For 
$x\notin\dom f$ we set $\hat{\partial}f(x):=\emptyset$. The {\it limiting (Mordukhovich) subdifferential} is defined at $x\in \dom f$ by 
$$\partial f(x)=\{v\in\R^m:\exists x_n\rightarrow x,f(x_n)\rightarrow f(x)\mbox{ and }\exists v_n\in\hat{\partial}f(x_n),v_n\rightarrow v \mbox{ as }n\rightarrow+\infty\},$$
while for $x \notin \dom f$, one takes $\partial f(x) :=\emptyset$.

Notice that in case $f$ is convex, these notions coincide with the {\it convex subdifferential}, which means that  
$\hat\partial f(x)=\partial f(x)=\{v\in\R^m:f(y)\geq f(x)+\<v,y-x\> \ \forall y\in \R^m\}$ for all $x\in\dom f$. 

Notice the inclusion $\hat\partial f(x)\subseteq\partial f(x)$ for each $x\in\R^m$. We will use the following closedness criteria 
concerning the graph of the limiting subdifferential: if $(x_n)_{n\in\N}$ and $(v_n)_{n\in\N}$ are sequences in $\R^m$ such that 
$v_n\in\partial f(x_n)$ for all $n\in\N$, $(x_n,v_n)\rightarrow (x,v)$ and $f(x_n)\rightarrow f(x)$ as $n\rightarrow+\infty$, then 
$v\in\partial f(x)$. 

The Fermat rule reads in this nonsmooth setting as: if $x\in\R^m$ is a local minimizer of $f$, then $0\in\partial f(x)$. Notice that 
in case $f$ is continuously differentiable around $x \in \R^m$ we have $\partial f(x)=\{\nabla f(x)\}$. Let us denote by 
$$\crit(f)=\{x\in\R^m: 0\in\partial f(x)\}$$ the set of {\it (limiting)-critical points} of $f$. Let us mention also the following subdifferential rule: 
if $f:\R^m\rightarrow(-\infty,+\infty]$ is proper and lower semicontinuous  and $h:\R^m\rightarrow \R$ is a continuously differentiable function, then $\partial (f+h)(x)=\partial f(x)+\nabla h(x)$ for all $x\in\R^m$. 

We turn now our attention to functions satisfying the {\it Kurdyka-\L{}ojasiewicz property}. This class of functions will play 
a crucial role when proving the convergence of the proposed inertial algorithm. For $\eta\in(0,+\infty]$, we denote by $\Theta_{\eta}$ the class of concave and continuous functions 
$\varphi:[0,\eta)\rightarrow [0,+\infty)$ such that $\varphi(0)=0$, $\varphi$ is continuously differentiable on $(0,\eta)$, continuous at $0$ and $\varphi'(s)>0$ for all 
$s\in(0, \eta)$. In the following definition (see \cite{att-b-red-soub2010, b-sab-teb}) we use also the {\it distance function} to a set, defined for $A\subseteq\R^m$ as $\dist(x,A)=\inf_{y\in A}\|x-y\|$  
for all $x\in\R^m$. 

\begin{definition}\label{KL-property} \rm({\it Kurdyka-\L{}ojasiewicz property}) Let $f:\R^m\rightarrow(-\infty,+\infty]$ be a proper and lower semicontinuous 
function. We say that $f$ satisfies the {\it Kurdyka-\L{}ojasiewicz (KL) property} at $\ol x\in \dom\partial f=\{x\in\R^m:\partial f(x)\neq\emptyset\}$ 
if there exists $\eta \in(0,+\infty]$, a neighborhood $U$ of $\ol x$ and a function $\varphi\in \Theta_{\eta}$ such that for all $x$ in the 
intersection 
$$U\cap \{x\in\R^m: f(\ol x)<f(x)<f(\ol x)+\eta\}$$ the following inequality holds 
$$\varphi'(f(x)-f(\ol x))\dist(0,\partial f(x))\geq 1.$$
If $f$ satisfies the KL property at each point in $\dom\partial f$, then $f$ is called a {\it KL function}. 
\end{definition}

The origins of this notion go back to the pioneering work of \L{}ojasiewicz \cite{lojasiewicz1963}, where it is proved that for a real-analytic function 
$f:\R^m\rightarrow\R$ and a critical point $\ol x\in\R^m$ (that is $\nabla f(\ol x)=0$), there exists $\theta\in[1/2,1)$ such that the function 
$|f-f(\ol x)|\|\nabla f\|^{-1}$ is bounded around $\ol x$. This corresponds to the situation when $\varphi(s)=s^{1-\theta}$. The result of 
\L{}ojasiewicz allows the interpretation of the KL property as a reparametrization of the function values in order to avoid flatness around the 
critical points. Kurdyka \cite{kurdyka1998} extended this property to differentiable functions definable in an o-minimal structure. 
Further extensions to the nonsmooth setting can be found in \cite{b-d-l2006, att-b-red-soub2010, b-d-l-s2007, b-d-l-m2010}. 

One of the remarkable properties of the KL functions is their ubiquitous in applications, according to \cite{b-sab-teb}. To the class of KL functions belong semi-algebraic, real sub-analytic, semiconvex, uniformly convex and 
convex functions satisfying a growth condition. We refer the reader to 
\cite{b-d-l2006, att-b-red-soub2010, b-d-l-m2010, b-sab-teb, b-d-l-s2007, att-b-sv2013, attouch-bolte2009} and the references therein  for more details regarding all the classes mentioned above and illustrating examples. 

An important role in our convergence analysis will be played by the following uniformized KL property given in \cite[Lemma 6]{b-sab-teb}. 

\begin{lemma}\label{unif-KL-property} Let $\Omega\subseteq \R^m$ be a compact set and let $f:\R^m\rightarrow(-\infty,+\infty]$ be a proper 
and lower semicontinuous function. Assume that $f$ is constant on $\Omega$ and $f$ satisfies the KL property at each point of $\Omega$.   
Then there exist $\varepsilon,\eta >0$ and $\varphi\in \Theta_{\eta}$ such that for all $\ol x\in\Omega$ and for all $x$ in the intersection 
\begin{equation}\label{int} \{x\in\R^m: \dist(x,\Omega)<\varepsilon\}\cap \{x\in\R^m: f(\ol x)<f(x)<f(\ol x)+\eta\}\end{equation} 
the following inequality holds \begin{equation}\label{KL-ineq}\varphi'(f(x)-f(\ol x))\dist(0,\partial f(x))\geq 1.\end{equation}
\end{lemma}

We close this section by presenting two convergence results which will play a determined role in the proof of the results we provide in the next section. The first one was often used in the literature
in the context of Fej\'{e}r monotonicity techniques for proving convergence results of classical algorithms for convex optimization problems or more generally for monotone inclusion problems (see \cite{bauschke-book}). 
The second one is probably also known, see for example \cite{b-c-inertial-nonc-ts}. 

\begin{lemma}\label{fejer1} Let $(a_n)_{n\in\N}$ and $(b_n)_{n\in\N}$ be real sequences such that $b_n\geq 0$ for all $n\in\N$, 
$(a_n)_{n\in\N}$ is bounded below and $a_{n+1}+b_n\leq a_n$ for all $n\in\N$. Then $(a_n)_{n\in\N}$ is a monotically decreasing and convergent 
sequence and $\sum_{n\in \N}b_n< + \infty$.
\end{lemma}

\begin{lemma}\label{fejer2} Let $(a_n)_{n\in\N}$ and $(b_n)_{n\in\N}$ be nonnegative real sequences, such that 
$\sum_{n\in\N}b_n<+\infty$ and $a_{n+1}\le a\cdot a_n+b\cdot a_{n-1} +b_n$ for all $n\ge 1$, where $a\in\R$, $b\ge 0$ and $a+b<1.$ Then $\sum_{n\in\N}a_n<+\infty.$
\end{lemma}

\section{A forward-backward algorithm}

In this section we present an inertial forward-backward algorithm for a fully nonconvex optimization problem and study its convergence properties. The problem under investigation has the following formulation. \vspace{0.2cm}

\noindent {\bf Problem 1.} Let $f:\R^m\to (-\infty,+\infty]$ be a proper, lower semicontinuous function which is bounded below and let $g:\R^m\to \R$ be a Fr\'{e}chet 
differentiable function with Lipschitz continuous gradient, i.e. there exists $L_{\n g}\ge 0$ such that $\|\n g(x)-\n g(y)\|\le L_{\n g}\|x-y\|$ for all $x,y\in\R^m.$ We deal with the optimization problem 
\begin{equation}\label{opt-pb} (P) \ \inf_{x\in\R^m}[f(x)+g(x)]. \end{equation}

In the iterative scheme we propose below, we use  also the function 
$F:\R^m\to\R$, assumed to be $\s-$strongly convex, i.e. $F - \frac{\sigma}{2}\|\cdot\|^2$ is convex, Fr\'{e}chet differentiable and such that $\n F$ is $L_{\n F}$-Lipschitz continuous, where $\s,L_{\n F}>0$. 
The {\it Bregman distance} to $F$, denoted by $D_F:\R^m\times\R^m\to \R$, is defined as
$$D_F(x,y)=F(x)-F(y)-\<\n F(y),x-y\> \ \forall (x,y)\in\R^m\times\R^m.$$

Notice that the properties of the function $F$ ensure the following inequalities 
\begin{equation}\label{l0}
\frac{\s}{2}\|x-y\|^2\le D_F(x,y)\le \frac{L_{\n F}}{2}\|x-y\|^2 \ \forall x,y\in\R^m.
\end{equation}

We propose the following iterative scheme. \vspace{0.2cm}

\noindent{\bf Algorithm 1.} Chose $x_0,x_1\in\R^m$, $\ul\a,\ol\a > 0$, $\b\geq 0$ and the sequences 
$(\a_n)_{n\geq 1},(\b_n)_{n\geq 1}$ fulfilling $$0<\ul\a\leq\a_n\leq\ol\a \ \forall n\geq 1$$ and 
$$0\leq\b_n\leq\b \ \forall n\geq 1.$$
 Consider the iterative scheme  
\begin{equation}\label{alg}(\forall n\geq 1)\hspace{0.2cm}
x_{n+1}\in\argmin_{u\in \R^m}\left\{D_F(u,x_n)+\a_n\<u,\n g(x_n)\>+\b_n\<u,x_{n-1}-x_n\>+\a_n f(u)\right\}.\end{equation}

Due to the subdfferential sum formula mentioned in the previous section, one can see that the sequence generated by this algorithm satisfies the relation 
\begin{equation}\label{alg1}
x_{n+1}\in(\n F+\a_n\p f)^{-1}(\n F(x_n)-\a_n\n g(x_n)+\b_n(x_n-x_{n-1})) \ \forall n\geq 1. 
\end{equation}

Further, since $f$ is proper, lower semicontinuous and bounded from below and $D_F$ is coercive in its first argument 
(that is $\lim_{\|x\|\rightarrow+\infty}D_F(x,y)=+\infty$ for all $y\in\R^m$), the iterative scheme is well-defined, meaning 
that the existence of $x_n$ is guaranteed for each $n \geq 2$, since the objective function in the minimization problem to be solved at each iteration is coercive. 

\begin{remark}\label{prox-nevid} The condition that $f$ should be bounded below is imposed in order to ensure that in each iteration one can chose at least one $x_n$ 
(that is the $\argmin$ in \eqref{alg} is nonempty). One can replace this requirement by asking that the objective function in the minimization problem 
considered in \eqref{alg} is coercive and the theory presented below still remains valid. This observation is useful when dealing with optimization problems as the ones considered in Subsection 4.1. 
\end{remark}

Before proceeding with the convergence analysis, we discuss the relation of our scheme to other algorithms from the literature. Let us take first $F(x)=\frac{1}{2}\|x\|^2$ for all $x\in\R^m$. 
In this case $D_F(x,y)=\frac{1}{2}\|x-y\|^2$ for all $(x,y)\in\R^m\times\R^m$ and $\sigma=L_{\nabla F}=1$. The iterative scheme becomes 
\begin{equation}\label{alg2}(\forall n\geq 1)\hspace{0.2cm}x_{n+1}\in\argmin_{u\in \R^m}\left\{\frac{\|u-(x_n-\a_n\n g(x_n)+\b_n(x_n-x_{n-1}))\|^2}{2\a_n}+f(u)\right\}.
\end{equation}
A similar inertial type algorithm has been analyzed in \cite{ipiano}, however in the restrictive case when $f$ is convex. 
If we take in addition $\b=0$, which enforces $\b_n=0$ for all $n\geq 1$, then \eqref{alg2} becomes 
\begin{equation}\label{alg3}(\forall n\geq 1)\hspace{0.2cm}x_{n+1}\in\argmin_{u\in \R^m}\left\{\frac{\|u-(x_n-\a_n\n g(x_n))\|^2}{2\a_n}+f(u)\right\},
\end{equation}
the convergence of which has been investigated in \cite{b-sab-teb} in the full nonconvex setting. Notice that forward-backward algorithms  
with variable metrics for KL functions have been proposed in \cite{f-g-peyp, c-pesquet-r}. 

On the other hand, if we take $g(x)=0$ for all $x\in\R^m$, the iterative scheme in \eqref{alg2} becomes 
\begin{equation}\label{alg4} \ 
(\forall n\geq 1)\hspace{0.2cm}x_{n+1}\in\argmin_{u\in \R^m}\left\{\frac{\|u-(x_n+\b_n(x_n-x_{n-1}))\|^2}{2\a_n}+f(u)\right\},
\end{equation} which is a proximal point algorithm with inertial/memory effects formulated in the nonconvex setting designed for finding the critical points 
of $f$. The iterative scheme without the inertial term, that is when $\b=0$ and, so, $\b_n=0$ for all $n \geq 1$, has been considered in the context of KL functions in \cite{attouch-bolte2009}. 

Let us mention that in the full convex setting, which means that $f$ and $g$ are convex
functions, in which case for all $n\geq 2$, $x_n$ is uniquely determined and can be expressed via the {\it proximal operator} of $f$, \eqref{alg2} can be 
derived from the iterative scheme proposed in \cite{moudafi-oliny2003}, \eqref{alg3} is the classical forward-backward algorithm (see for example 
\cite{bauschke-book} or \cite{combettes}) and \eqref{alg4} has been analyzed in \cite{alvarez-attouch2001} in the more general context of 
monotone inclusion problems. 

In the convergence analysis of the algorithm the following result will be useful (see for example \cite[Lemma 1.2.3]{nes}).

\begin{lemma}\label{l3} Let $g:\R^m\to\R$ be Fr\'echet differentiable with $L_{\n g}$-Lipschitz continuous gradient. Then
$$g(y)\le g(x)+\<\n g(x),y-x\>+\frac{L_{\n g}}{2}\|y-x\|^2,\,\forall x,y\in\R^m.$$
\end{lemma}

\noindent Let us start now with the investigation of the convergence of the proposed algorithm. 

\begin{lemma}\label{l4} In the setting of Problem 1, let $(x_n)_{n\in\N}$ be the sequence generated by Algorithm 1. Then for every $\m>0$ one has
$$(f+g)(x_{n+1})+M_1\|x_n-x_{n+1}\|^2\le (f+g)(x_{n})+M_2\|x_{n-1}-x_{n}\|^2 \ \forall n\geq 1,$$
where 
\begin{equation}\label{M1M2}
M_1=\ds\frac{\s-\oa L_{\n g}}{2\oa}-\frac{\m\b}{2\ua} \ \mbox{and} \ M_2=\ds\frac{\b}{2\m\ua}.
\end{equation}

Moreover, for $\m>0$ and $\ua,\b$ satisfying
\begin{equation}\label{mu-sigma}\m(\s-L_{\n g}\ua)>\b(\m^2+1)\end{equation} one can chose $\ua < \oa$ such that $M_1>M_2$.
\end{lemma}

\begin{proof} Let us consider $\m>0$ and fix $n\ge 1$. Due to \eqref{alg} we have 
$$D_F(x_{n+1},x_n)+\a_n\<x_{n+1},\n g(x_n)\>+\b_n\<x_{n+1},x_{n-1}-x_n\>+\a_n f(x_{n+1})\le$$
  $$D_F(x_n,x_n)+\a_n\<x_n,\n g(x_n)\>+\b_n\<x_n,x_{n-1}-x_n\>+\a_n f(x_n)$$
or, equivalently,
\begin{equation}\label{e2}
D_F(x_{n+1},x_n)+\<x_{n+1}-x_n,\a_n\n g(x_n)-\b_n(x_n-x_{n-1})\>+\a_n f(x_{n+1})\le \a_n f(x_{n}).
\end{equation}
On the other hand, by Lemma \ref{l3} we have
$$\<\n g(x_n),x_{n+1}-x_n\>\ge g(x_{n+1})-g(x_n)-\frac{L_{\n g}}{2}\|x_n-x_{n+1}\|^2.$$
At the same time
$$\<x_{n+1}-x_n,x_{n-1}-x_n\>\ge -\left(\frac{\m}{2}\|x_n-x_{n+1}\|^2+\frac{1}{2\m}\|x_{n-1}-x_n\|^2\right),$$
and from (\ref{l0}) we have
$$\frac{\s}{2}\|x_{n+1}-x_n\|^2\le D_F(x_{n+1},x_n).$$
Hence, $(\ref{e2})$ leads to
\begin{equation}\label{e3}
(f+g)(x_{n+1})+\frac{\s-L_{\n g}\a_n-\m\b_n}{2\a_n}\|x_{n+1}-x_n\|^2\le (f+g)(x_n)+\frac{\b_n}{2\m\a_n}\|x_{n-1}-x_n\|^2.
\end{equation}
Obviously $M_1=\frac{\s-L_{\n g}\oa}{2\oa}-\frac{\m\b}{2\ua}\le \frac{\s-L_{\n g}\a_n-\m\b_n}{2\a_n}$ and $M_2=\frac{\b}{2\m\ua}\ge \frac{\b_n}{2\m\a_n}$ thus,
$$(f+g)(x_{n+1})+M_1\|x_n-x_{n+1}\|^2\le (f+g)(x_{n})+M_2\|x_{n-1}-x_{n}\|^2$$
and the first part of the lemma is proved. 

Let now $\m>0$ and $\ua,\b$ be such that $\m(\s-L_{\n g}\ua)>\b(\m^2+1)$. Then 
$$\frac{\m\ua\s}{L_{\n g}\m\ua+\b(\m^2+1)}>\ua.$$
Let $$\ua<\oa<\frac{\m\ua\s}{L_{\n g}\m\ua+\b(\m^2+1)}.$$
Then $$\frac{1}{\oa}>\frac{L_{\n g}}{\s}+\frac{\b(\m^2+1)}{\m\ua\s}\Leftrightarrow \frac{\s-L_{\n g}\oa}{2\oa}>\frac{\b(\m^2+1)}{2\m\ua}\Leftrightarrow \frac{\s-L_{\n g}\oa}{2\oa}-\frac{\m\b}{2\ua}>\frac{\b}{2\m\ua}\Leftrightarrow$$
$$M_1>M_2$$ and the proof is complete. 
\end{proof}

\begin{proposition}\label{p1} In the setting of Problem 1, chose $\m,\ua, \b$ satisfying \eqref{mu-sigma}, $M_1, M_2$ satisfying \eqref{M1M2} and $\ua < \oa$ such that $M_1 > M_2$. Assume that $f+g$ is bounded from below.
Then the following statements hold:
\begin{itemize}
\item[(a)] $\sum_{n\geq 1}\|x_n-x_{n-1}\|^2<+\infty$;
\item[(b)] the sequence $((f+g)(x_n)+M_2\|x_{n-1}-x_n\|^2)_{n\geq 1}$ is monotonically decreasing and convergent; 
\item[(c)] the sequence $((f+g)(x_n))_{n\in\N}$ is convergent.
\end{itemize}
\end{proposition}

\begin{proof} For every $n\geq 1$, set $a_n=(f+g)(x_n)+M_2\|x_{n-1}-x_n\|^2$ and $ b_n=(M_1-M_2)\|x_n-x_{n+1}\|^2.$ Then obviously from Lemma \ref{l4} one has 
for every $n\geq 1$ 
$$a_{n+1}+b_n=(f+g)(x_{n+1})+M_1\|x_n-x_{n+1}\|^2\le (f+g)(x_n)+M_2\|x_{n-1}-x_n\|^2=a_n.$$
The conclusion follows now from Lemma \ref{fejer1}. 
\end{proof}

\begin{lemma}\label{subdiff} In the setting of Problem 1, consider the sequences generated by Algorithm 1. For every $n \geq 1$ we have
\begin{equation}\label{y-n-in-subdiff}y_{n+1}\in\partial (f+g)(x_{n+1}),\end{equation}
where 
$$y_{n+1} = \frac{\n F(x_{n})-\n F(x_{n+1})}{\a_{n}}+\n g(x_{n+1})-\n g(x_{n})+\frac{\b_{n}}{\a_{n}}(x_{n}-x_{n-1}).$$
Moreover, 
\begin{equation}\label{ineq-y_n}\|y_{n+1}\|\leq \frac{L_{\n F}+\a_{n}L_{\n g}}{\a_{n}}\|x_{n}-x_{n+1}\|+\frac{\b_n}{\a_n}\|x_{n}-x_{n-1}\| \ \forall n\geq 1\end{equation}
\end{lemma}

\begin{proof} Let us fix $n\geq 1$. From (\ref{alg1}) we have that
$$\frac{\n F(x_{n})-\n F(x_{n+1})}{\a_{n}}-\n g(x_{n})+\frac{\b_{n}}{\a_{n}}(x_{n}-x_{n-1})\in\p f(x_{n+1}),$$
or, equivalently,
$$y_{n+1}-\n g(x_{n+1})\in \p f(x_{n+1}),$$
which shows that $y_{n+1}\in\p(f+g)(x_{n+1}).$

The inequality \eqref{ineq-y_n} follows now from the definition of $y_{n+1}$ and the triangle inequality.  
\end{proof}

\begin{lemma}\label{l5} In the setting of Problem 1, chose $\m,\ua, \b$ satisfying \eqref{mu-sigma}, $M_1, M_2$ satisfying \eqref{M1M2} and $\ua < \oa$ such that $M_1 > M_2$. Assume that $f+g$ is coercive, i.e.
$$\lim_{\|x\|\to+\infty}(f+g)(x)=+\infty.$$
Then the sequence $(x_n)_{n\in\N}$ generated by Algorithm 1 has a subsequence convergent to a critical point of $f+g.$ Actually every cluster point 
of $(x_n)_{n\in\N}$ is a critical point of $f+g.$
\end{lemma}

\begin{proof} Since $f+g$ is a proper, lower semicontinuous and coercive function, it follows that $\inf_{x\in\R^m}[f(x)+g(x)]$ is finite and the infimum is attained. 
Hence $f+g$ is bounded from below. 

(i) According to Proposition \ref{p1}(b), we have $$(f+g)(x_n)\leq (f+g)(x_n)+M_2\|x_n-x_{n-1}\|^2\leq (f+g)(x_1)+M_2\|x_1-x_0\|^2 \ \forall n\geq 1.$$
Since the function $f+g$ is coercive, its lower level sets are bounded, thus the sequence  $(x_n)_{n\in\N}$ is bounded. 

Let $x$ be a cluster point of $(x_n)_{n\in\N}$. Then there exists a subsequence $(x_{n_k})_{k\in\N}$ such that $x_{n_k}\rightarrow \ol x$ as $k\rightarrow+\infty$. 
We show that $(f+g)(x_{n_k})\to (f+g)(x)$ as $k\to+\infty$ and that $x$ is a critical point 
of $f+g$, that is $0\in\p(f+g)(x).$

We show first that $f(x_{n_k})\to f(x)$ as $k\to+\infty.$
Since $f$ is lower semicontinuous one has
$$\liminf_{k\to+\infty}f(x_{n_k})\ge f(x).$$
On the other hand, from (\ref{alg}) we have for every $n \geq 1$
\begin{align*}
D_F(x_{n+1},x_n)+\a_n\<x_{n+1},\n g(x_n)\>+\b_n\<x_{n+1},x_{n-1}-x_n\>+\a_n f(x_{n+1}) & \le\\
D_F(x,x_n)+\a_n\<x,\n g(x_n)\>+\b_n\<x,x_{n-1}-x_n\>+\a_n f(x), &
\end{align*}
which leads to
\begin{align*}  
\frac{1}{\a_{n_k-1}}\left(D_F(x_{n_k},x_{n_k-1})-D_F(x,x_{n_k-1}) \right) + &\\
\frac{1}{\a_{n_k-1}}\left(\<x_{n_k}-x,\a_{n_k-1}\n g(x_{n_k-1})-\b_{n_k-1}(x_{n_k-1}-x_{n_k-2})\>\right)+&\\
f(x_{n_k}) & \le f(x) \ \forall k\geq 2.
\end{align*}
The latter combined with Proposition \ref{p1}(a) and \eqref{l0} shows that $\limsup_{k\to+\infty}f(x_{n_k})\le f(x)$, 
hence $\lim_{k\to +\infty} f(x_{n_k})=f(x).$ Since $g$ is continuous, obviously $g(x_{n_k})\to g(x)$ as $k\to+\infty,$ thus 
$(f+g)(x_{n_k})\to (f+g)(x)$ as $k\to+\infty.$

Further, by using the notations from Lemma \ref{subdiff}, we have $y_{n_k}\in\p f(x_{n_k})$ for every $k \geq 2$. 
By Proposition \ref{p1}(a) and Lemma \ref{subdiff} we get $y_{n_k}\to 0$ as $k\to+\infty.$

Concluding, we have:
$$y_{n_k}\in\p (f+g)(x_{n_k}) \ \forall k \geq 2,$$
$$(x_{n_k},y_{n_k})\to (x,0),\,k\to +\infty$$
$$(f+g)(x_{n_k})\to (f+g)(x),\,k\to +\infty.$$
Hence $0\in\p(f+g)(x),$ that is, $x$ is a critical point of $f+g.$
\end{proof}

\begin{lemma}\label{l6} In the setting of Problem 1, chose $\m,\ua, \b$ satisfying \eqref{mu-sigma}, $M_1, M_2$ satisfying \eqref{M1M2} and $\ua < \oa$ such that $M_1 > M_2$. Assume that $f+g$ is coercive and
consider the function
$$H:\R^m\times\R^m\to(-\infty,+\infty],\, H(x,y)=(f+g)(x)+M_2\|x-y\|^2 \ \forall (x,y)\in\R^m\times\R^m.$$
Let $(x_n)_{n\in\N}$ be the sequence generated by Algorithm 1. Then there exist $M,N>0$ such that the following statements hold:
\begin{itemize}
\item[($H_1$)] $H(x_{n+1},x_{n})+M\|x_{n+1}-x_{n}\|^2\le H(x_{n},x_{n-1})$ for all $n\geq 1$; 
\item[($H_2$)] for all $n\geq 1,$ there exists $w_{n+1}\in \p H(x_{n+1},x_n)$ such that $\|w_{n+1}\|\le N(\|x_{n+1}-x_n\|+\|x_n-x_{n-1}\|)$;
\item[($H_3$)] if $(x_{n_k})_{k\in\N}$ is a subsequence such that $x_{n_k}\to x$ as $k\to+\infty$, then $H(x_{n_k}, x_{n_k-1})\to H(x,x)$ as $k\to+\infty$ 
(there exists at least one subsequence with this property). 
\end{itemize}
\end{lemma}

\begin{proof} For ($H_1$) just take $M=M_1-M_2$ and the conclusion follows from Lemma \ref{l4}.

Let us prove ($H_2$). For every $n\geq 1$ we define $$w_{n+1}=(y_{n+1}+2M_2(x_{n+1}-x_n),2M_2(x_n-x_{n+1})),$$
where $(y_n)_{n \geq 2}$ is the sequence introduced in Lemma \ref{subdiff}. 
The fact that $w_{n+1}\in \p H(x_{n+1},x_n)$ follows from Lemma \ref{subdiff} and the relation
\begin{equation}\label{H-subdiff}\partial H(x,y)=\big(\partial (f+h)(x)+2M_2(x-y)\big)\times \{2M_2(y-x)\} \ \forall (x,y)\in\R^m\times\R^m.\end{equation}

Further, one has (see also Lemma \ref{subdiff}) $$\|w_{n+1}\|\le \|y_{n+1}+2M_2(x_{n+1}-x_n)\|+\|2M_2(x_n-x_{n+1})\|\le $$
$$\left(\frac{L_{\n F}}{\a_n}+L_{\n g}+4M_2\right)\|x_{n+1}-x_n\|+\frac{\b_n}{\a_n}\|x_n-x_{n-1}\|.$$
Since $0<\ua\le\a_n\le\oa$ and $0\le\b_n\le\b$ for all $n\geq 1$, one can chose 
$$N=\sup_{n\geq 1}\left\{\frac{L_{\n F}}{\a_n}+L_{\n g}+4M_2, \frac{\b_n}{\a_n}\right\}<+\infty$$ and the conclusion follows.

For ($H_3$), consider $(x_{n_k})_{k\in\N}$ a subsequence such that $x_{n_k}\to x$ as $k\to+\infty$. 
We have shown in the proof of Lemma \ref{l5} that $(f+g)(x_{n_k})\to (f+g)(x)$ as $k\to+\infty.$ From Proposition \ref{p1}(a) and 
the definition of $H$ we easily derive that $H(x_{n_k}, x_{n_k-1})\to H(x,x)=(f+g)(x)$ as $k\to+\infty.$ The existence of such a sequence follows from 
Lemma \ref{l5}.
\end{proof}

In the following we denote by $\omega((x_n)_{n\in\N})$ the set of cluster points of the sequence $(x_n)_{n\in\N}$.

\begin{lemma}\label{l7}  In the setting of Problem 1, chose $\m,\ua, \b$ satisfying \eqref{mu-sigma}, $M_1, M_2$ satisfying \eqref{M1M2} and $\ua < \oa$ such that $M_1 > M_2$. Assume that $f+g$ is coercive and
consider the function
$$H:\R^m\times\R^m\to(-\infty,+\infty],\, H(x,y)=(f+g)(x)+M_2\|x-y\|^2 \ \forall (x,y)\in\R^m\times\R^m.$$
Let $(x_n)_{n\in\N}$ be the sequence generated by Algorithm 1. Then the following statements are true:
\begin{itemize}
\item[(a)] $\omega((x_n,x_{n-1})_{n\geq 1})\subseteq \crit(H)=\{(x,x)\in\R^m\times\R^m:x\in \crit(f+g)\}$; 
\item[(b)] $\lim_{n\to\infty}\dist((x_n,x_{n-1}),\omega((x_n,x_{n-1}))_{n\geq 1})=0$;
\item[(c)] $\omega((x_n,x_{n-1})_{n\geq 1})$ is nonempty, compact and connected;
\item[(d)] $H$ is finite and constant on $\omega((x_n,x_{n-1})_{n\geq 1}).$
\end{itemize}
\end{lemma}

\begin{proof} (a) According to Lemma \ref{l5} and Proposition \ref{p1}(a) we have 
$\omega((x_n,x_{n-1})_{n\geq 1})\subseteq \{(x,x)\in\R^m\times\R^m:x\in \crit(f+g)\}.$ The equality $\crit(H)=\{(x,x)\in\R^m\times\R^m:x\in \crit(f+g)\}$ 
follows from \eqref{H-subdiff}. 

(b) and (c) can be shown as in \cite[Lemma 5]{b-sab-teb}, by also taking into consideration \cite[Remark 5]{b-sab-teb}, where it is noticed 
that the properties (b) and (c) are generic for sequences satisfying $x_{n+1}-x_n\rightarrow 0$ as $n\rightarrow+\infty$.

(d) According to  Proposition \ref{p1}, the sequence $((f+g)(x_n))_{n\in\N}$ is convergent, i.e. $\lim_{n\to+\infty}(f+g)(x_n)=l\in\R.$ Take an arbitrary 
$(x,x)\in\omega((x_n,x_{n-1})_{n\geq 1})$, where $x\in\crit(f+g)$ (we took statement (a) into consideration). From Lemma \ref{l6}($H_3$) it follows that there exists a 
subsequence $(x_{n_k})_{k\in\N}$ such that $x_{n_k}\to x$ as $k\to+\infty$ and 
$H(x_{n_k}, x_{n_k-1})\to H(x,x)$ as $k\to+\infty$. Moreover, from Proposition \ref{p1} one has 
$H(x,x)=\lim_{k\to+\infty}H(x_{n_k}, x_{n_k-1})=\lim_{k\to+\infty}(f+g)(x_{n_k})+M_2\|x_{n_k}-x_{n_k-1}\|^2=l$ and the conclusion follows.
\end{proof}

We give now the main result concerning the convergence of the whole sequence $(x_n)_{n\in\N}$. 

\begin{theorem}\label{t1} In the setting of Problem 1, chose $\m,\ua, \b$ satisfying \eqref{mu-sigma}, $M_1, M_2$ satisfying \eqref{M1M2} and $\ua < \oa$ such that $M_1 > M_2$. Assume that $f+g$ is coercive and that
$$H:\R^m\times\R^m\to(-\infty,+\infty],\, H(x,y)=(f+g)(x)+M_2\|x-y\|^2 \ \forall (x,y)\in\R^m\times\R^m$$ 
is a KL function. Let $(x_n)_{n\in\N}$ be the sequence generated by Algorithm 1. Then the following statements are true:\begin{itemize}
 \item[(a)] $\sum_{n\in\N}\|x_{n+1}-x_n\|<+\infty$;
 \item[(b)] there exists $x\in\crit(f+g)$ such that $\lim_{n\rightarrow+\infty}x_n=x$.
\end{itemize}
\end{theorem}

\begin{proof} (a) According to Lemma \ref{l7} we can consider an element $\ol x\in\crit(f+g)$ such that $(\ol x, \ol x) \in \omega ((x_n,x_{n-1})_{n\geq 1})$. 
In analogy to the proof of Lemma \ref{l6} (by taking into account also the decrease property (H1)) one can
easily show that $\lim_{n\rightarrow +\infty}H(x_n,x_{n-1})=H(\ol x,\ol x)$. We separately treat the following two cases. 

I. There exists $\ol n\in\N$ such that $H(x_{\ol n},x_{\ol n-1})=H(\ol x,\ol x)$. The decrease property in Lemma \ref{l6}(H1) implies 
$H(x_{n},x_{n-1})=H(\ol x,\ol x)$ for every $n\geq \ol n$. One can show inductively that the sequence $(x_n,x_{n-1})_{n\geq \ol n}$ is constant and the 
conclusion follows. 

II. For all $n\geq 1$ we have $H(x_n,x_{n-1})>H(\ol x,\ol x)$. Take $\Omega:=\omega ((x_n,x_{n-1})_{n\geq 1})$.

In virtue of Lemma \ref{l7}(c) and (d) and Lemma \ref{unif-KL-property}, the KL property of $H$ leads to the existence of positive numbers $\e$ and $\eta$ and 
a concave function $\varphi\in\Phi_{\eta}$ such that for all
\begin{align}\label{int-H} 
(x,y)\in & \{(u,v)\in\R^m\times\R^m: \dist((u,v),\Omega)<\e\} \nonumber \\ 
 & \cap\{(u,v)\in\R^m\times\R^m:H(\ol x,\ol x)<H(u,v)<H(\ol x,\ol x)+\eta\}\end{align}
one has
\begin{equation}\label{ineq-H}\varphi'(H(x,y)-H(\ol x,\ol x))\dist((0,0),\p H(x,y))\ge 1.\end{equation}

Let $n_1\in\N$ such that $H(x_n,x_{n-1})<H(\ol x,\ol x)+\eta$ for all $n\geq n_1.$  
According to Lemma \ref{l7}(b), there exists $n_2\in \N$ such that $\dist((x_n,x_{n-1}),\Omega)<\e$ for all $n\ge n_2.$

Hence the sequence $(x_n,x_{n-1})_{n\geq \ol n}$ where $\ol n=\max\{n_1,n_2\}$, belongs to the intersection \eqref{int-H}. So we have (see \eqref{ineq-H})
 $$\varphi'(H(x_n,x_{n-1})-H(\ol x,\ol x))\dist((0,0),\p H(x_n,x_{n-1}))\ge 1 \ \forall n\geq\ol n.$$
Since $\varphi$ is concave, it holds 
\begin{align*}
\varphi(H(x_n,x_{n-1})-H(\ol x,\ol x))-\varphi(H(x_{n+1},x_{n})-H(\ol x,\ol x)) & \ge \\ 
\varphi'(H(x_n,x_{n-1})-H(\ol x,\ol x))\cdot(H(x_n,x_{n-1})-H(x_{n+1},x_{n})) & \ge \\
\frac{H(x_n,x_{n-1})-H(x_{n+1},x_{n})}{\dist((0,0),\p H(x_n,x_{n-1}))} & \ \forall n\geq \ol n.
\end{align*}

Let $M,N > 0$ be the real numbers furnished by Lemma \ref{l6}. According to Lemma \ref{l6}($H_2$) there exists $w_{n}\in \p H(x_{n},x_{n-1})$ such that $\|w_{n}\|\le N(\|x_{n}-x_{n-1}\|+\|x_{n-1}-x_{n-2}\|)$ for all 
$n\geq 2$. Then obviously $\dist((0,0),\p H(x_n,x_{n-1}))\le\|w_n\|,$ hence
\begin{align*}
\varphi(H(x_n,x_{n-1})-H(x^0,x^0))-\varphi(H(x_{n+1},x_{n})-H(x^0,x^0)) & \ge \\
\frac{H(x_n,x_{n-1})-H(x_{n+1},x_{n})}{\|w_n\|} & \ge \\
\frac{H(x_n,x_{n-1})-H(x_{n+1},x_{n})}{N(\|x_{n}-x_{n-1}\|+\|x_{n-1}-x_{n-2}\|)} & \ \forall n\geq \ol n.
\end{align*}
On the other hand, from Lemma \ref{l6}($H_1$) we obtain that $$H(x_n,x_{n-1})-H(x_{n+1},x_{n})\ge M\|x_{n+1}-x_{n}\|^2 \ \forall n\geq 1.$$
Hence, one has
\begin{align*}
\varphi(H(x_n,x_{n-1})-H(x^0,x^0))-\varphi(H(x_{n+1},x_{n})-H(x^0,x^0)) & \ge \\ 
\frac{M\|x_{n+1}-x_{n}\|^2}{N(\|x_{n}-x_{n-1}\|+\|x_{n-1}-x_{n-2}\|)} & \ \forall n\geq \ol n.
\end{align*}

For all $n\geq 1$, let us denote $\frac{N}{M}(\varphi(H(x_n,x_{n-1})-H(\ol x,\ol x))-\varphi(H(x_{n+1},x_{n})-H(\ol x,\ol x)))=\e_n$ and $\|x_{n}-x_{n-1}\|=a_n.$
Then the last inequality becomes
\begin{equation}\label{ineq-e-a}\e_n\ge\frac{a_{n+1}^2}{a_n+a_{n-1}} \ \forall n\geq\ol n.\end{equation}

Obviously, since $\varphi\geq 0$, for $S\geq 1$ we have 
$\sum_{n=1}^S\e_n=(N/M)(\varphi(H(x_1,x_{0})-H(\ol x,\ol x))-\varphi(H(x_{S+1},x_{S})-H(\ol x,\ol x)))\leq (N/M)(\varphi(H(x_1,x_{0})-H(\ol x,\ol x)))$, 
hence $\sum_{n\geq 1}\e_n<+\infty.$

On the other hand, from \eqref{ineq-e-a} we derive 
$$a_{n+1}=\sqrt{\e_n(a_n+a_{n-1})}\le \frac{1}{4}(a_n+a_{n-1})+\e_n \ \forall n\geq\ol n.$$
Hence, according to Lemma \ref{fejer2}, $\sum_{n\geq 1}a_n<+\infty$, that is $\sum_{n\in\N}\|x_n-x_{n+1}\|<+\infty.$ 

(b) It follows from (a) that $(x_n)_{n\in\N}$ is a Cauchy sequence, hence it is convergent. Applying Lemma \ref{l5}, there exists $x\in\crit(f+g)$ such that 
$\lim_{n\rightarrow+\infty}x_n=x$.
\end{proof}

Since the class of semi-algebraic functions is closed under addition (see for example \cite{b-sab-teb}) and 
$(x,y) \mapsto c\|x-y\|^2$ is semi-algebraic for $c>0$, we obtain also the following direct consequence.  

\begin{corollary}\label{cor-f+g} In the setting of Problem 1, chose $\m,\ua, \b$ satisfying \eqref{mu-sigma}, $M_1, M_2$ satisfying \eqref{M1M2} and $\ua < \oa$ such that $M_1 > M_2$. Assume that $f+g$ is coercive and 
semi-algebraic. Let $(x_n)_{n\in\N}$ be the sequence generated by Algorithm 1. Then the following statements are true:
\begin{itemize}
 \item[(a)] $\sum_{n\in\N}\|x_{n+1}-x_n\|<+\infty$;
 \item[(b)] there exists $x\in\crit(f+g)$ such that $\lim_{n\rightarrow+\infty}x_n=x$.
\end{itemize}
\end{corollary}

\begin{remark}\label{necoercive} As one can notice by taking a closer look at the proof of Lemma \ref{l5}, the conclusion of this statement as the ones of
Lemma \ref{l6}, Lemma \ref{l7}, Theorem \ref{t1} and Corollary \ref{cor-f+g} remain true, if instead of imposing that $f+g$ is coercive, we assume that
$f+g$ is bounded from below and the sequence $(x_n)_{n\in\N}$ generated by Algorithm 1 is bounded. This observation is useful when dealing with 
optimization problems as the ones considered in Subsection 4.2. 
\end{remark}

\section{Numerical experiments}

This section is devoted to the presentation of two numerical experiments which illustrate the applicability of the algorithm proposed in this work. In both numerical experiments
we considered $F = \frac{1}{2}\|\cdot\|^2$ and set $\mu=\sigma=1$.

\subsection{Detecting minimizers of nonconvex optimization problems}

As emphasized in \cite[Section 5.1]{ipiano} and \cite[Exercise 1.3.9]{bertsekas} one of the aspects which makes algorithms with inertial/memory effects
useful is given by the fact that they are able to detect optimal solutions of minimization problems which cannot be found by their non-inertial variants. 
In this subsection we show that this phenomenon arises even when solving problems of type \eqref{ex-opt-pb}, where the nonsmooth function $f$ is nonconvex. A similar situation has
been addressed in \cite{ipiano}, however, by assuming that $f$ is convex.

\begin{figure}[tb]	
	\centering
	\captionsetup[subfigure]{position=top}
	\subfloat[Contour plot]{\includegraphics*[viewport= 78 210 540 580, width=0.48\textwidth]{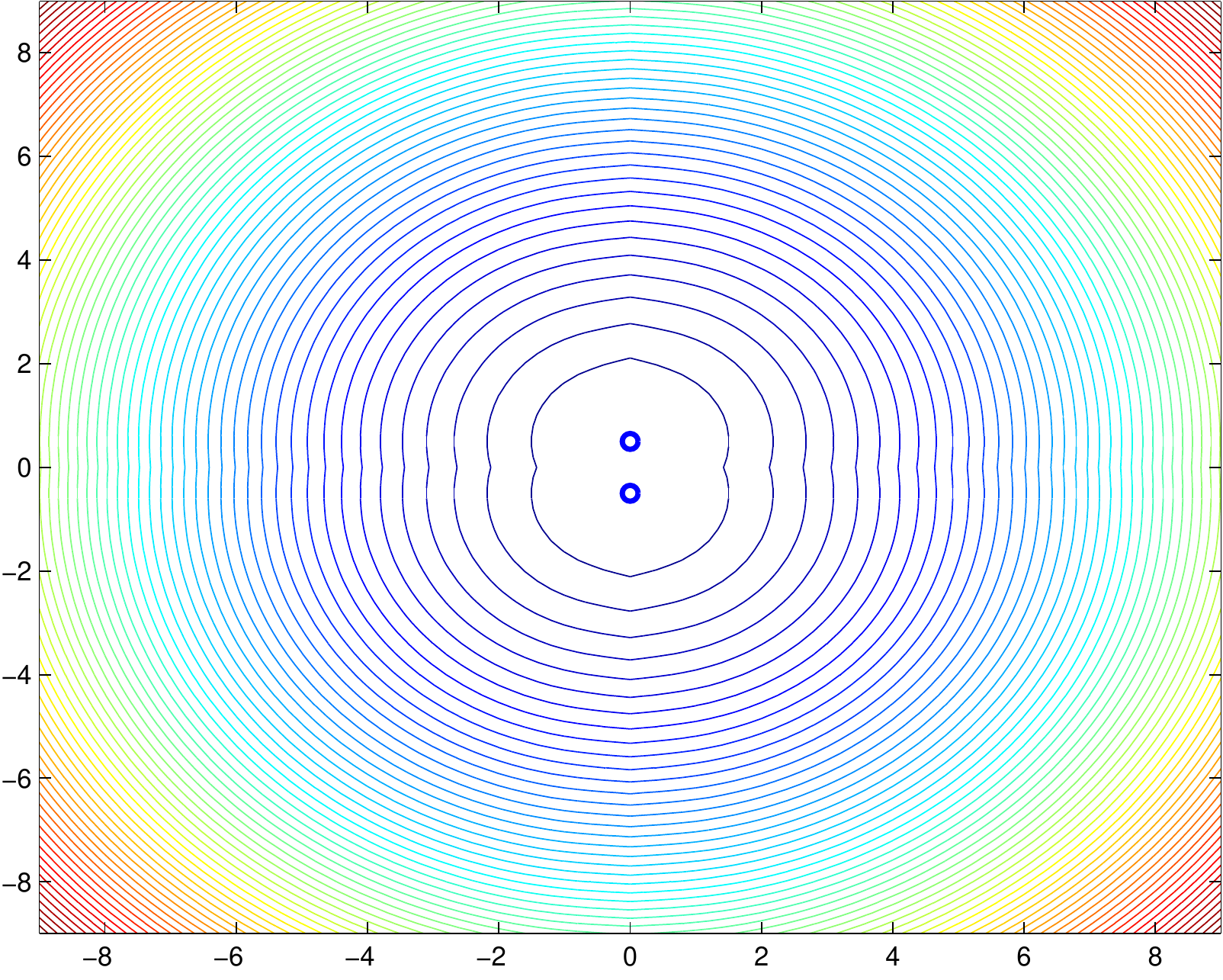}} \hspace{0.3mm}
	\subfloat[Graph]{\includegraphics*[viewport= 8 149 595 644, width=0.48\textwidth]{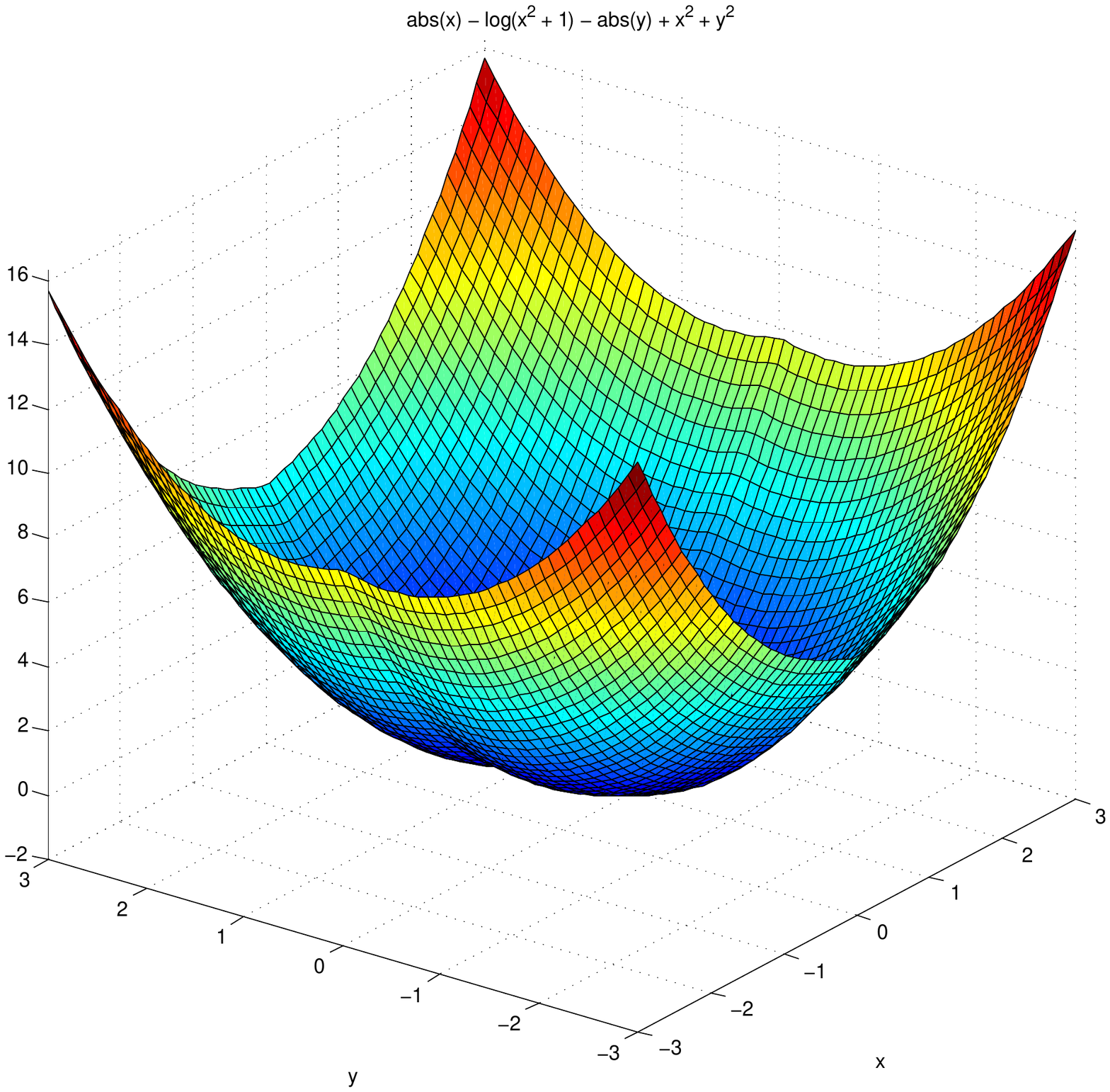}}  \hspace{0.3mm}
	\caption{\small Contour plot and graph of the objective function in \eqref{ex-opt-pb}. The two global optimal 
	solutions $(0,0.5)$ and $(0,-0.5)$ are marked on the first image.}
	\label{fig:lev-3d}	
\end{figure}	

Consider the optimization problem 
\begin{equation}\label{ex-opt-pb} \ \inf_{(x_1,x_2)\in\R^2} |x_1|-|x_2| + x_1^2-\log(1+x_1^2)+x_2^2 . \end{equation}

\begin{figure}[H]	
	\centering
	\captionsetup[subfigure]{position=top}
	\subfloat[$x_0=(-8,-8), \beta =0$]{\includegraphics*[viewport= 124 250 490 550, width=0.32\textwidth]{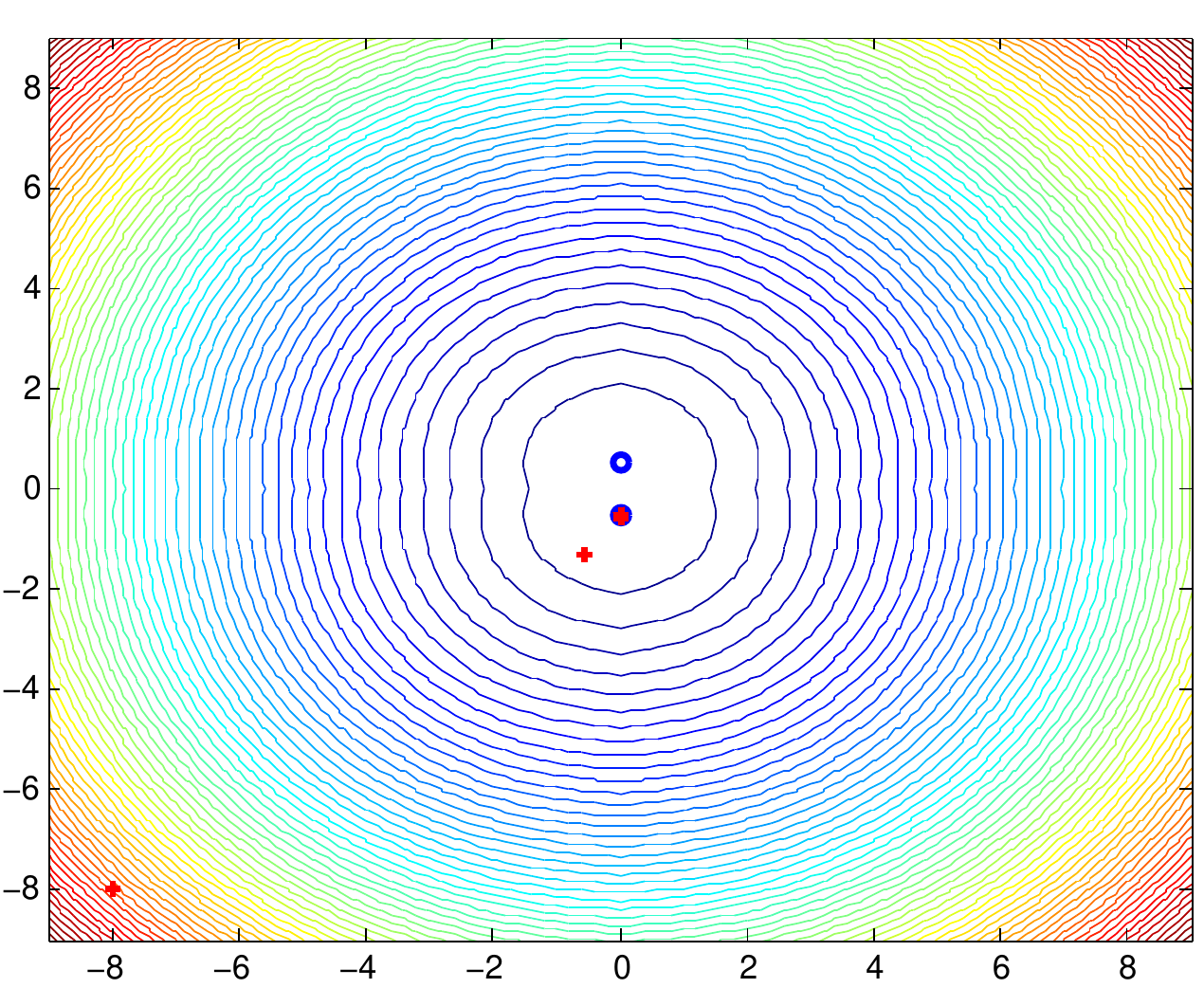}} \hspace{0.3mm}
	\subfloat[$x_0=(-8,-8), \beta =1.99$]{\includegraphics*[viewport= 124 250 490 550, width=0.32\textwidth]{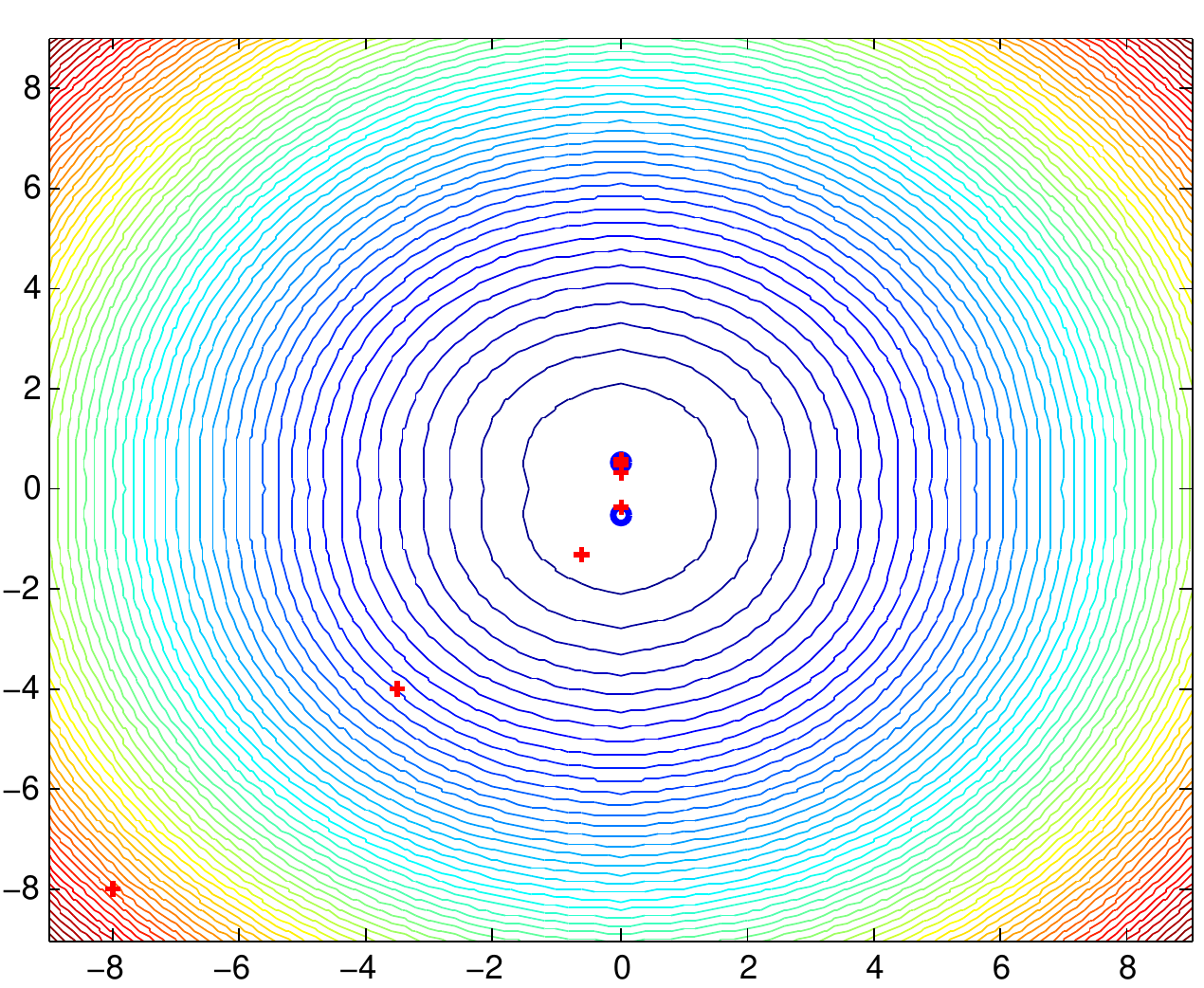}}  \hspace{0.3mm}
	\subfloat[$x_0=(-8,-8), \beta =2.99$]{\includegraphics*[viewport= 124 250 490 550, width=0.32\textwidth]{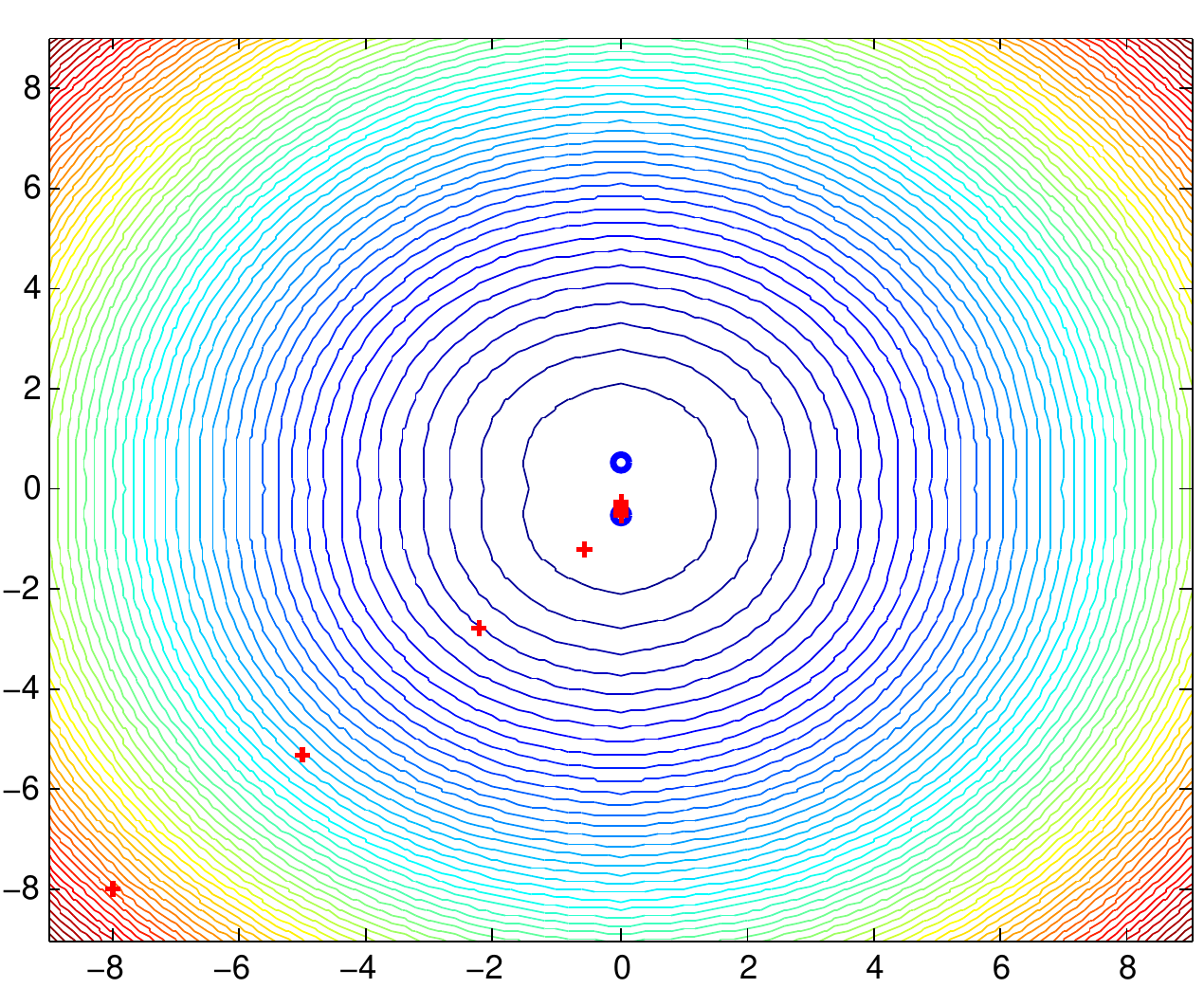}}\\
	\subfloat[$x_0=(-8,8), \beta =0$]{\includegraphics*[viewport= 124 250 490 550, width=0.32\textwidth]{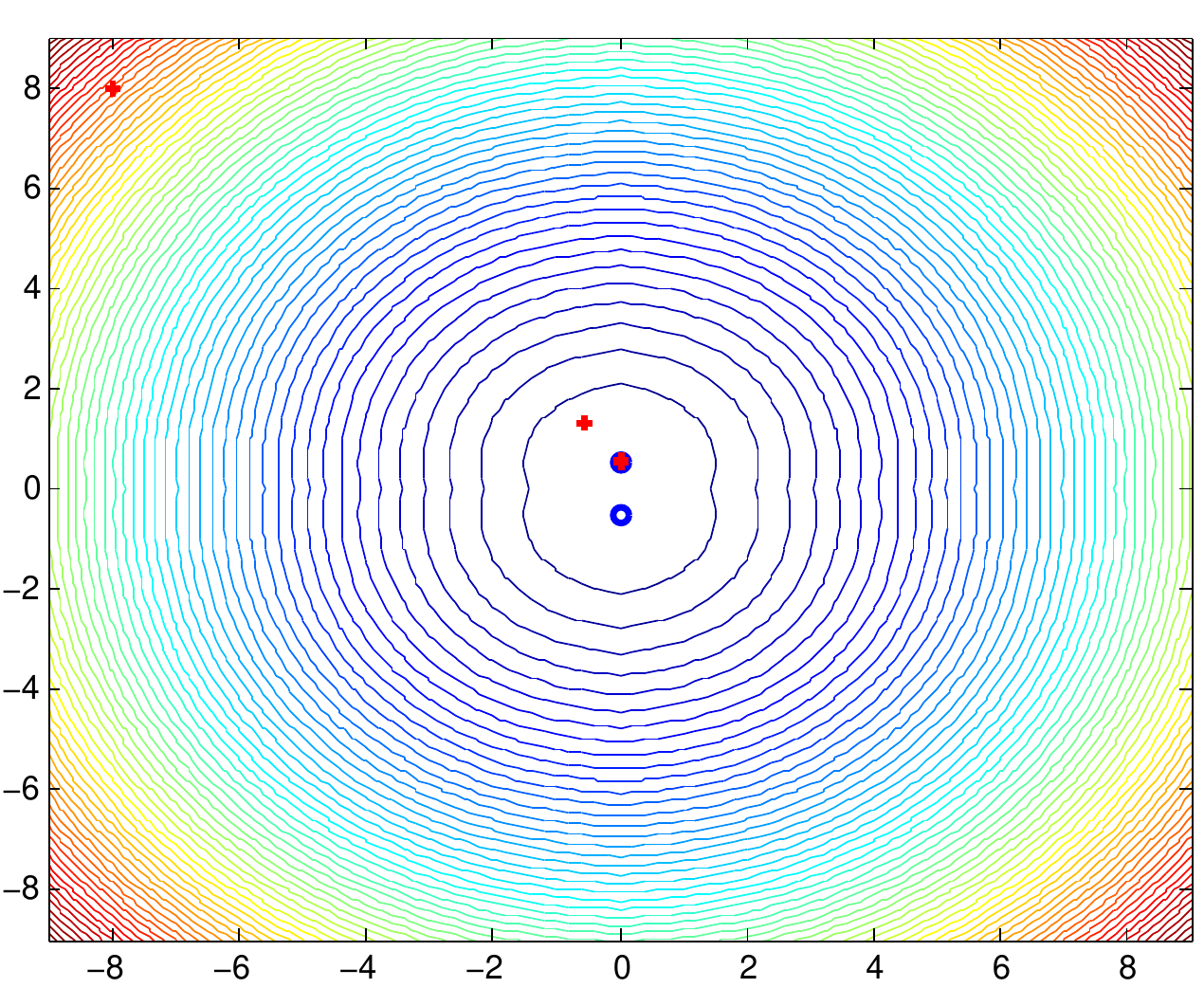}} \hspace{0.3mm}
	\subfloat[$x_0=(-8,8), \beta =1.99$]{\includegraphics*[viewport= 124 250 490 550, width=0.32\textwidth]{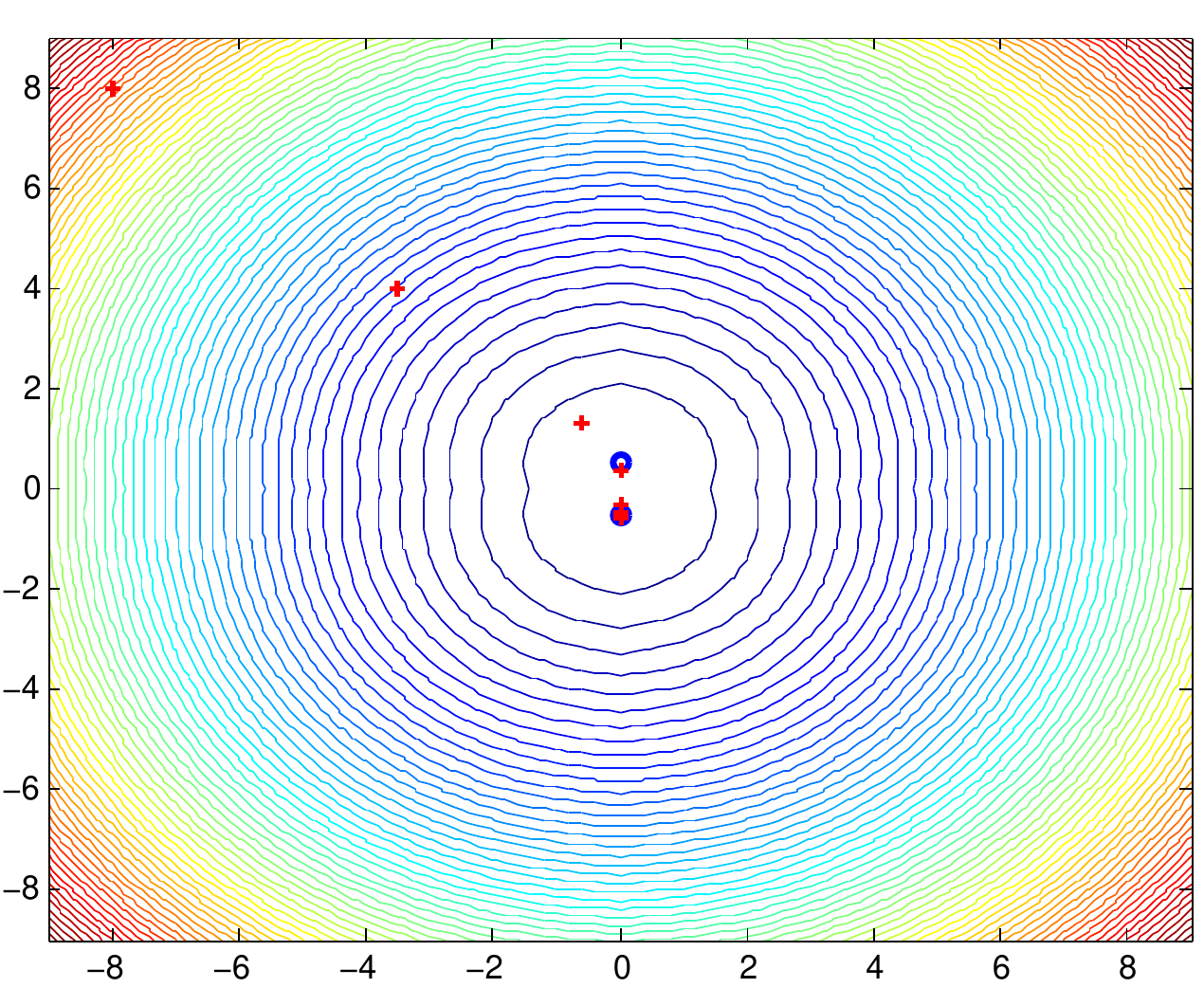}}  \hspace{0.3mm}
	\subfloat[$x_0=(-8,8), \beta =2.99$]{\includegraphics*[viewport= 124 250 490 550, width=0.32\textwidth]{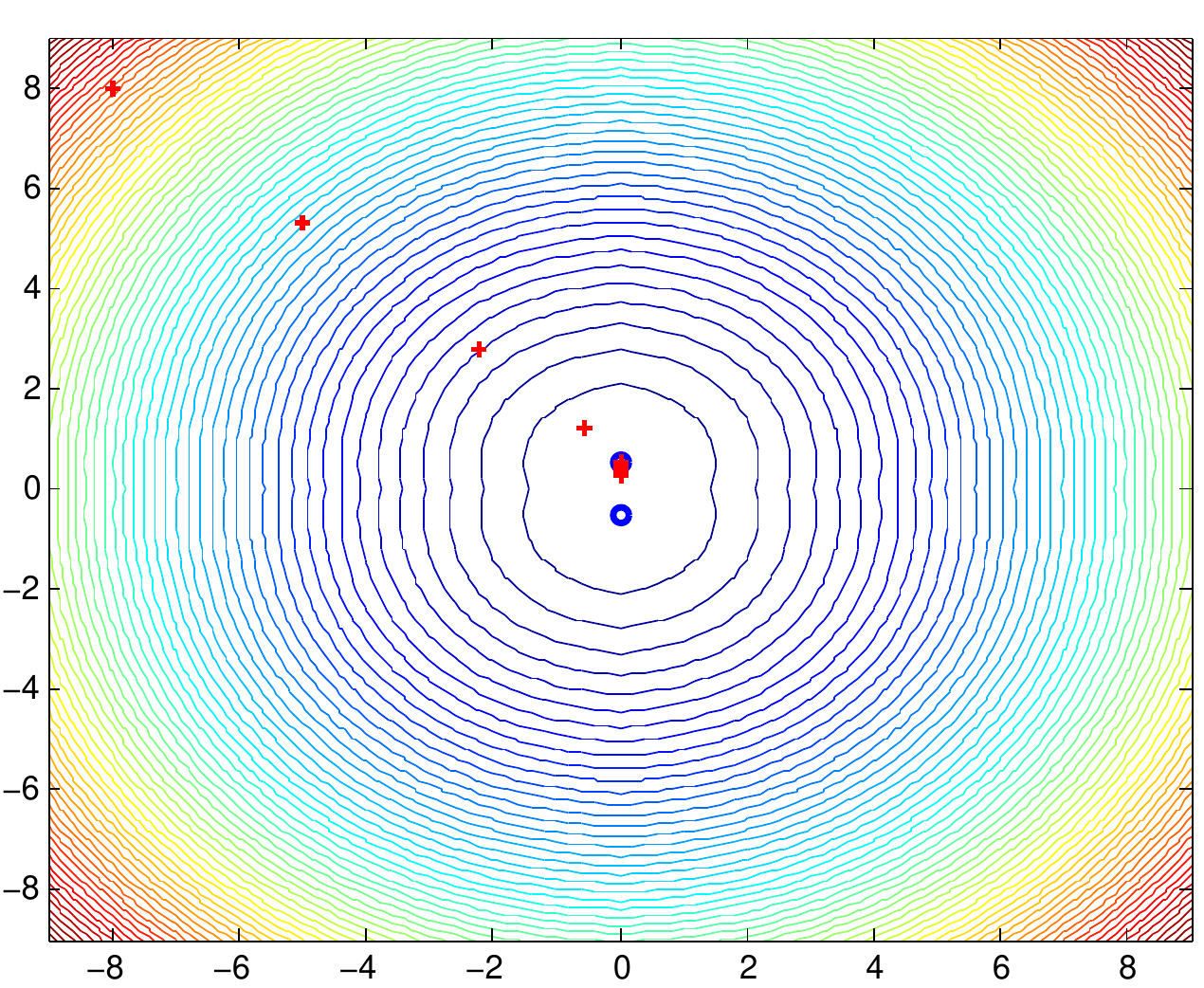}}\\
	\subfloat[$x_0=(8,-8), \beta =0$]{\includegraphics*[viewport= 124 250 490 550, width=0.32\textwidth]{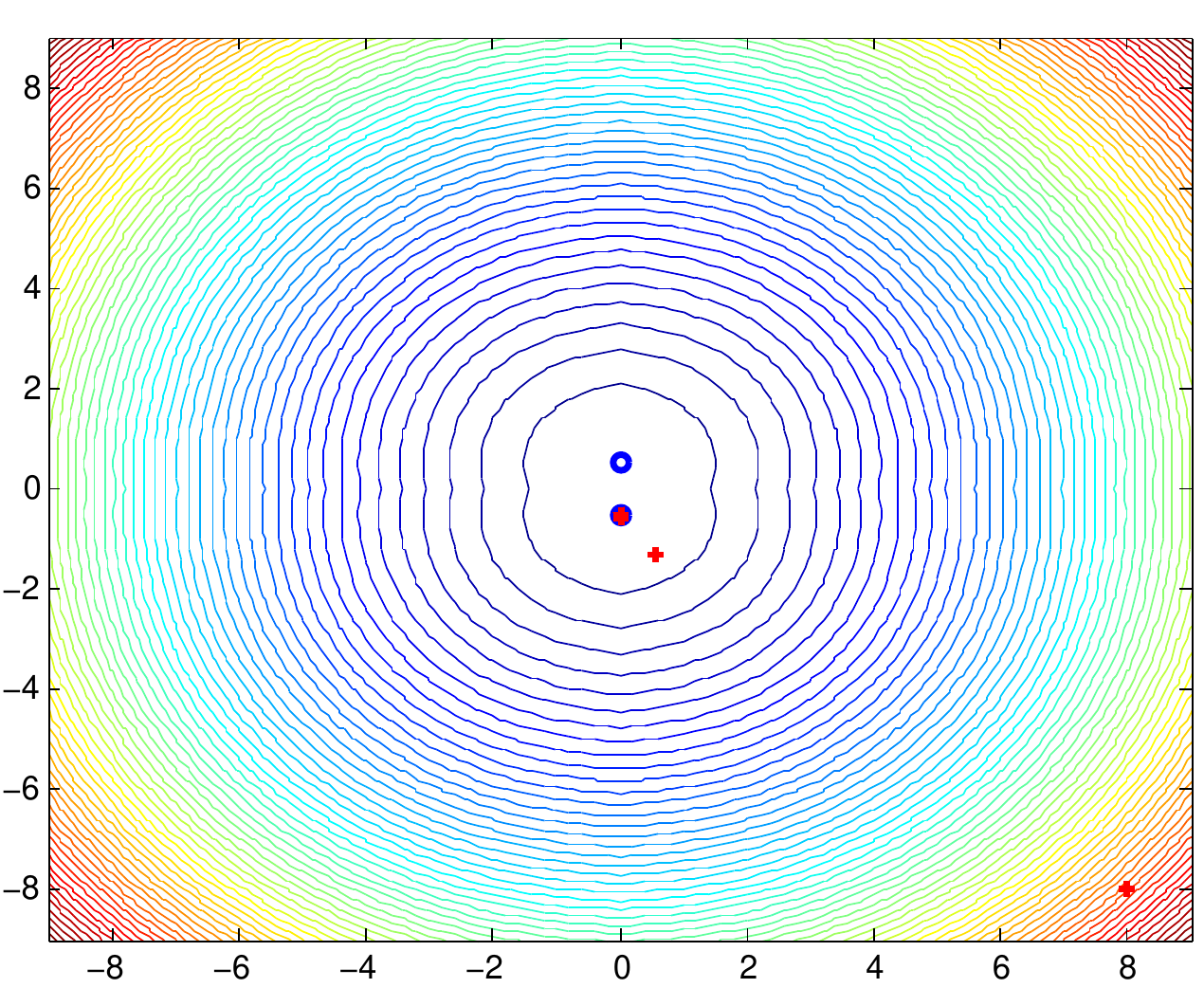}} \hspace{0.3mm}
	\subfloat[$x_0=(8,-8), \beta =1.99$]{\includegraphics*[viewport= 124 250 490 550, width=0.32\textwidth]{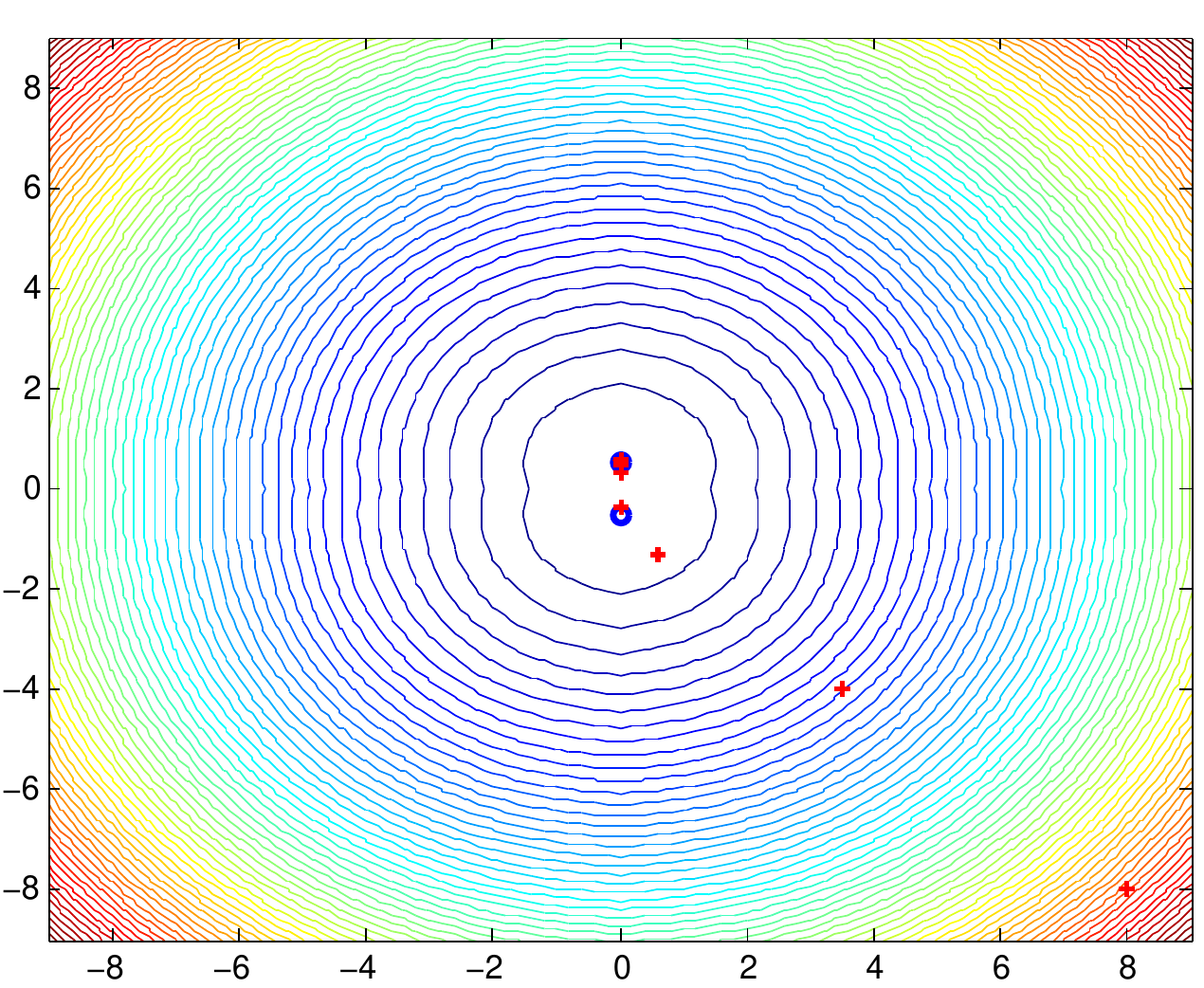}}  \hspace{0.3mm}
	\subfloat[$x_0=(8,-8), \beta =2.99$]{\includegraphics*[viewport= 124 250 490 550, width=0.32\textwidth]{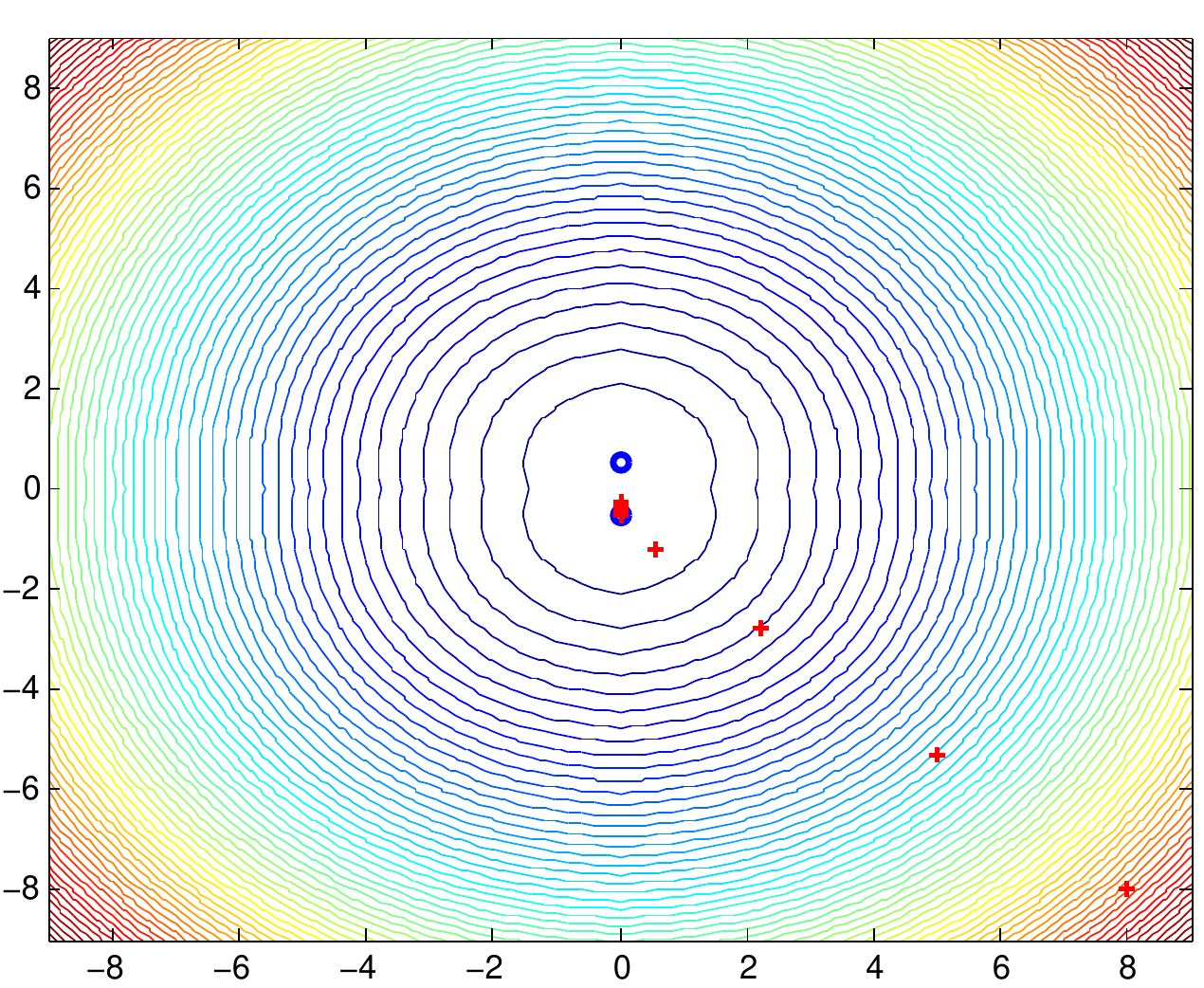}}\\
	\subfloat[$x_0=(8,8), \beta =0$]{\includegraphics*[viewport= 124 250 490 550, width=0.32\textwidth]{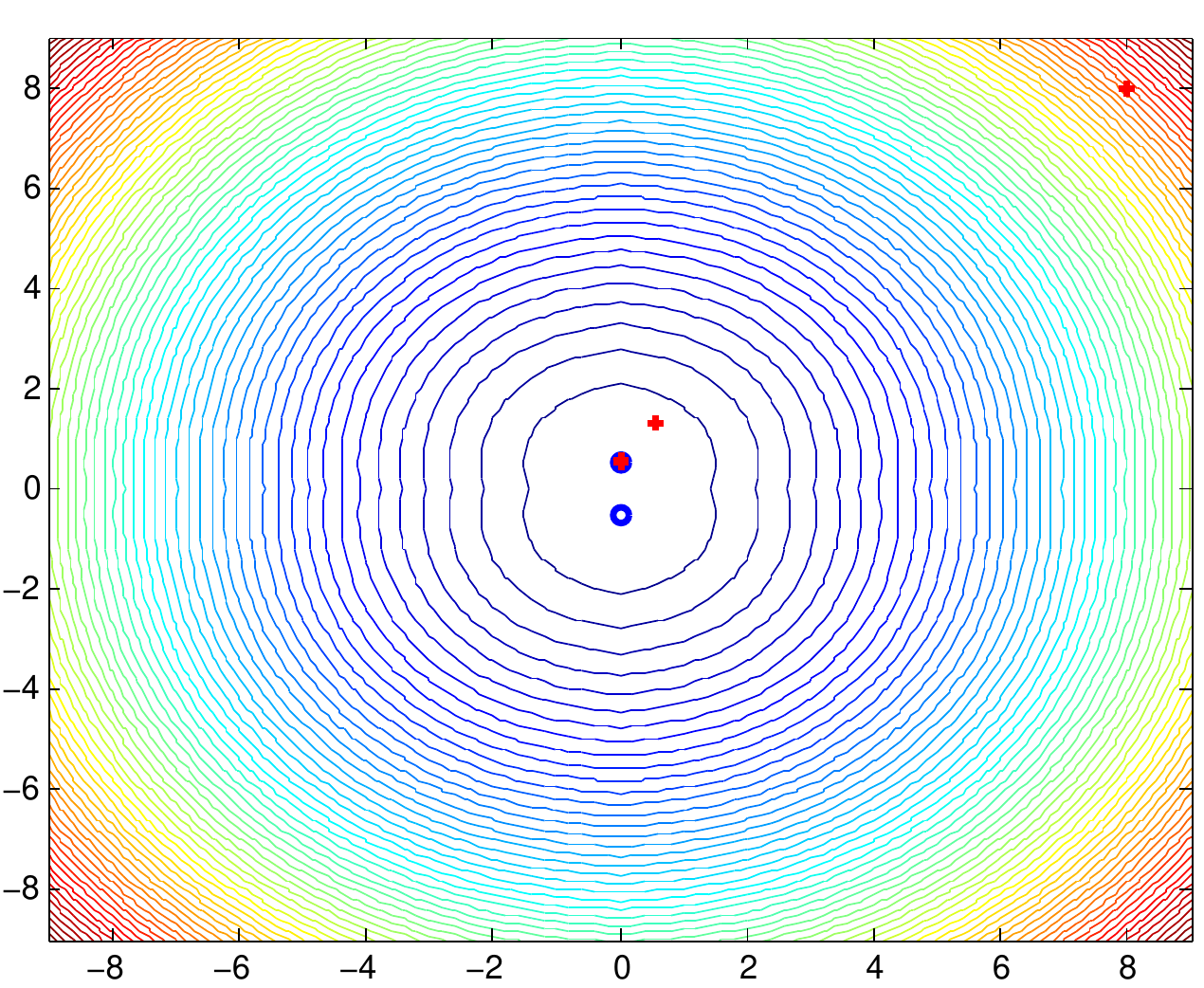}} \hspace{0.3mm}
	\subfloat[$x_0=(8,8), \beta =1.99$]{\includegraphics*[viewport= 124 250 490 550, width=0.32\textwidth]{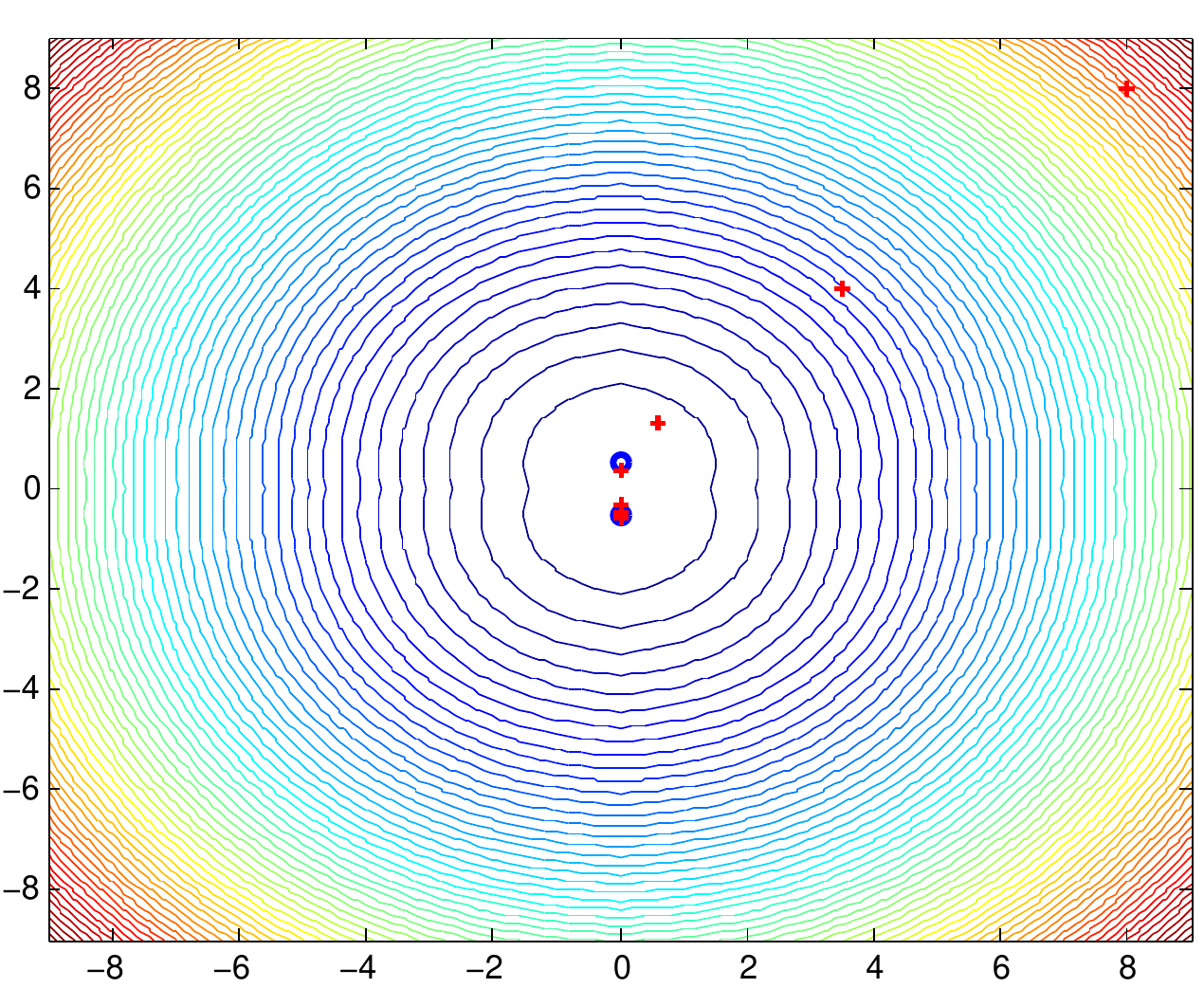}}  \hspace{0.3mm}
	\subfloat[$x_0=(8,8), \beta =2.99$]{\includegraphics*[viewport= 124 250 490 550, width=0.32\textwidth]{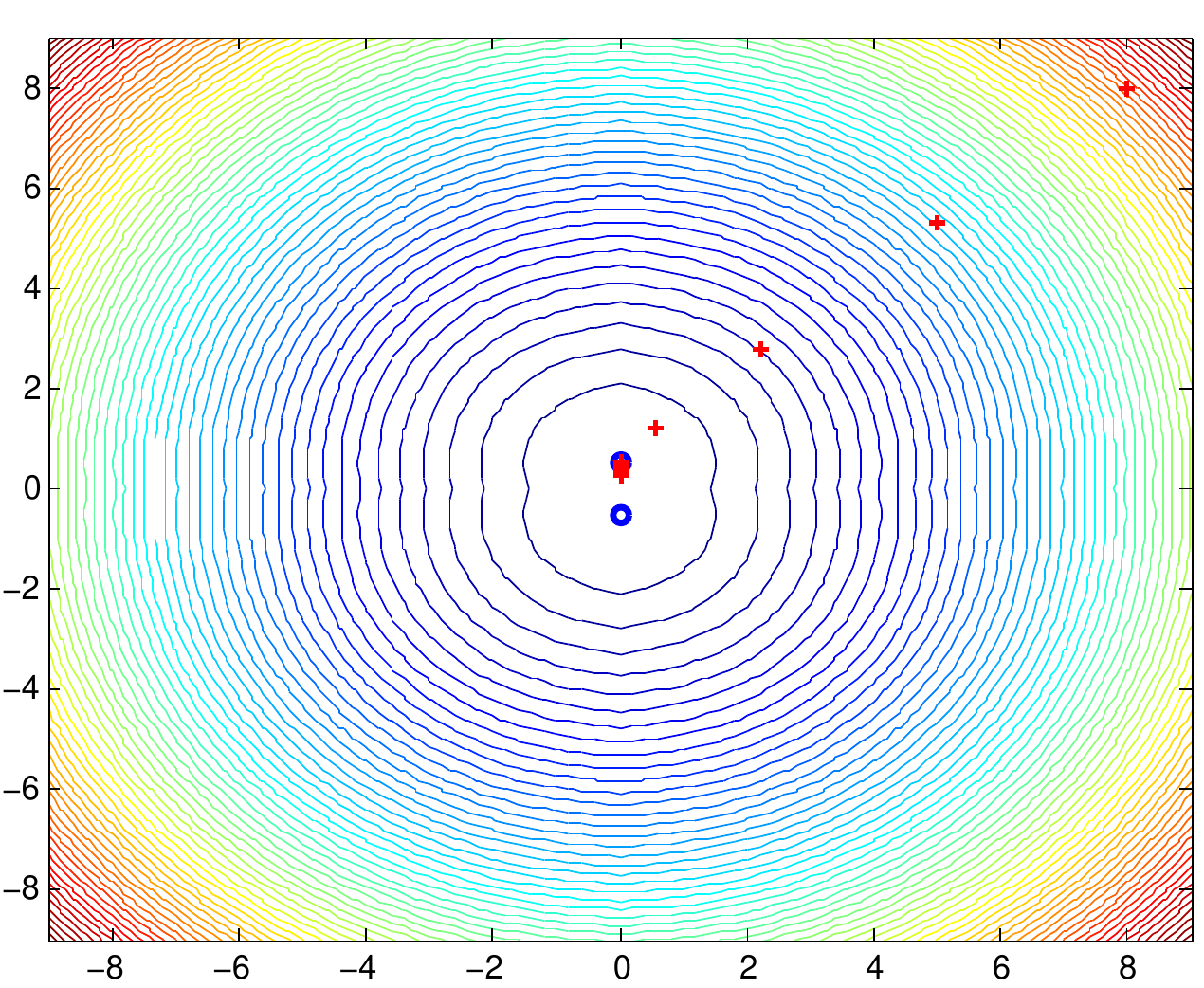}}
	\caption{\small Algorithm 1 after $100$ iterations and with starting points $(-8,-8), (-8,8), (8,-8)$ and $(8,8)$, respectively: the first column shows the iterates of the non-inertial version ($\beta_n=\beta = 0$ for all 
	$n\geq 1$), the second column the ones of the  inertial version with $\beta_n= \beta = 1.99$ for all $n\geq 1$ and the third column the ones of the inertial version with $\beta_n=\beta = 2.99$ for all $n\geq 1$.}
	\label{fig:mm}	
\end{figure}

The function $f : \R^2\rightarrow\R, f(x_1,x_2)= |x_1|-|x_2|,$ is nonconvex and continuous, the function $g : \R^2\rightarrow\R, g(x_1,x_2)=x_1^2-\log(1+x_1^2)+x_2^2,$ 
is continuously differentiable with Lipschitz continuous gradient with Lipschitz constant $L_{\nabla g}=9/4$ and one can easily prove that $f+g$ is coercive. Furthermore, combining \cite[the remarks after Definition 4.1]{att-b-red-soub2010}, \cite[Remark 5(iii)]{b-d-l-s2007} and 
\cite[Section 5: Example 4 and Theorem 3]{b-sab-teb}, one can easily conclude that $H$ in Theorem \ref{t1} is a KL function. 
By considering the first order optimality conditions 
$$-\nabla g(x_1,x_2)\in\partial f(x_1,x_2)=\partial (|\cdot|)(x_1)\times\partial (-|\cdot|)(x_2)$$ 
and by noticing that for all $x\in\R$ we have $$\partial(|\cdot|)(x)=\left\{
\begin{array}{ll}
1, & \mbox {if } x>0\\
-1, & \mbox {if } x<0\\
$[-1,1]$, & \mbox {if } x=0
\end{array}\right. \mbox{ and } \partial(-|\cdot|)(x)=\left\{
\begin{array}{ll}
-1, & \mbox {if } x>0,\\
1, & \mbox {if } x<0,\\
\{-1,1\}, & \mbox {if } x=0,
\end{array}\right.$$
(for the latter, see for example \cite{boris-carte}), one can easily determine the two critical points $(0,1/2)$ and $(0,-1/2)$ of \eqref{ex-opt-pb}, which are actually 
both optimal solutions of this minimization problem. In Figure \ref{fig:lev-3d}  the level sets and the graph of the objective function in \eqref{ex-opt-pb} are represented.

For $\gamma>0$ and $x=(x_1,x_2)\in\R^2$ we have (see Remark \ref{prox-nevid})
$$\prox\nolimits_{\gamma f}(x)=\argmin_{u\in \R^2}\left\{\frac{\|u-x\|^2}{2\gamma}+f(u)\right\} =\prox\nolimits_{\gamma |\cdot|}(x_1)\times \prox\nolimits_{\gamma (-|\cdot|)}(x_2),$$ 
where in the first component one has the well-known shrinkage operator
$$\prox\nolimits_{\gamma |\cdot|}(x_1)=x_1-\sgn(x_1)\cdot\min\{|x_1|,\gamma\},$$ while 
for the proximal operator in the second component the following formula can be proven
$$\prox\nolimits_{\gamma (-|\cdot|)}(x_2)=\left\{
\begin{array}{ll}
x_2+\gamma, & \mbox {if } x_2>0\\
x_2-\gamma, & \mbox {if } x_2<0\\
\{-\gamma,\gamma\}, & \mbox {if } x_2=0.
\end{array}\right.$$

We implemented Algorithm 1 by choosing $\beta_n=\beta=0$ for all $n\geq 1$ (which corresponds to the non-inertial version), $\beta_n=\beta=0.199$ for all $n \geq 1$ and 
$\beta_n=\beta=0.299$ for all $n\geq 1$, respectively, and by setting $\alpha_n=(0.99999-2\beta_n)/L_{\nabla g}$ for all $n\geq 1$. As starting points we considered the
corners of the box generated by the points $(\pm 8,\pm 8)$. Figure \ref{fig:mm} shows that independently of the four starting points we have the following phenomenon: the non-inertial version 
recovers only one of the two optimal solutions, situation which persists even when changing the value of $\alpha_n$; on the other hand, the inertial version is capable to 
find both optimal solutions, namely, one for $\beta = 0.199$  and the other one for $\beta=0.299$.

\subsection{Restoration of noisy blurred images}

The following numerical experiment concerns the restoration of a noisy blurred image by using a nonconvex misfit functional with nonconvex regularization.
For a given matrix $A \in \mathbb{R}^{m \times m}$ describing a blur operator and a
given vector $b \in \R^m$ representing the blurred and noisy image, the task is to estimate the unknown original image
$\ol x\in\R^m$ fulfilling
$$A\ol x=b.$$
To this end we solve the following regularized nonconvex minimization problem
 \begin{equation}\label{probimageproc-stud-zer-norm-wav}
\inf_{x \in \R^m} \left\{ \sum_{k=1}^M\sum_{l=1}^N\varphi\big((Ax-b)_{kl}\big) + \lambda \|Wx\|_0  \right\},
\end{equation}
where $\varphi:\R\rightarrow\R$, $\varphi(t)=\log(1+t^2),$ is derived form the Student's t distribution, 
$\lambda >0$ is a regularization parameter, 
$W:\R^m\rightarrow\R^m$ is a discrete Haar wavelet transform with 
four levels and $\|y\|_0=\sum_{i=1}^m|y_i|_0$ ($|\cdot|_0 = |\sgn(\cdot)|$) furnishes the number of nonzero entries of the vector $y=(y_1,...,y_m)\in\R^m$. In this context, $x \in \R^m$
represents the vectorized image $X\in\R^{M\times N}$, where $m = M\cdot N$ and $x_{i,j}$
denotes the normalized value of the pixel located in the $i$-th row and the $j$-th column, for
$i=1,\ldots,M$ and $j=1,\ldots,N$.

It is immediate that \eqref{probimageproc-stud-zer-norm-wav} can be written in the form \eqref{opt-pb}, by defining 
$f(x)= \lambda \|Wx\|_0$ and $g(x)= \sum_{k=1}^M\sum_{l=1}^N\varphi\big((Ax-b)_{kl}\big)$ for all $x\in\R^m$. 
By using that $W W^* = W^* W = I_m$, one can prove the following formula concerning the proximal operator of $f$
$$\prox\nolimits_{\gamma f}(x)=W^*\prox\nolimits_{\lambda\gamma\|\cdot\|_0}(Wx) \ \forall x \in \R^m \ \forall \gamma >0,$$
where for all $u=(u_1,...,u_m)$ we have (see \cite[Example 5.4(a)]{att-b-sv2013})
$$\prox\nolimits_{\lambda\gamma\|\cdot\|_0}(u)=(\prox\nolimits_{\lambda\gamma|\cdot|_0}(u_1),...,\prox\nolimits_{\lambda\gamma|\cdot|_0}(u_m))$$
and for all $t\in\R$ $$\prox\nolimits_{\lambda\gamma|\cdot|_0}(t)=\left\{
\begin{array}{ll}
t, & \mbox {if } |t|>\sqrt{2\lambda\gamma},\\
\{0,t\}, & \mbox {if } |t|=\sqrt{2\lambda\gamma},\\
0, & \mbox {otherwise.} 
\end{array}\right.$$
For the experiments we used the $256 \times 256$ boat test image which we first blurred by using a Gaussian 
blur operator of size $9 \times 9$ and standard deviation $4$ and to which we afterward added a zero-mean white 
Gaussian noise with standard deviation $10^{-6}$. In the first row of Figure \ref{fig:boat}  the original boat test image and the blurred and noisy one are represented, while in the second row
one has the reconstructed images by means of the non-inertial (for $\beta_n = \beta = 0$ for all $n\geq 1$)  and inertial versions (for $\beta_n= \beta = 10^{-7}$ for all $n \geq 1$) of Algorithm 1, respectively. 
We took as regularization parameter $\lambda=10^{-5}$ and set $\alpha_n=(0.999999 - 2\beta_n)/L_{\nabla g}$ for all $n \geq 1$, whereby the Lipschitz constant of the gradient of the smooth misfit function is
$L_{\nabla g} = 2$.

We compared the quality of the recovered images for $\beta_n = \beta$ for all $n\geq 1$ and different values of $\beta$ by making use of the improvement in signal-to-noise ratio 
(ISNR), which is defined as
$$ \text{ISNR}(n) = 10 \log_{10}\left( \frac{\left\|x-b\right\|^2}{\left\|x-x_n\right\|^2} \right),$$
where $x$, $b$ and $x_n$ denote the original, observed and estimated image at iteration $n$, respectively.

In the table below we list the values of the ISNR-function after $300$ iterations, whereby the case $\beta = 0$ corresponds to the non-inertial version of the algorithm.
One can notice that for $\beta$ taking very small values, the inertial version is competitive with the non-inertial one (actually it slightly outperforms it). 

\begin{table}
\begin{tabular}{ccccccc}
\toprule
$\beta$ &  $0.4$ & $0.2$ & $0.01$ & $0.0001$ & $10^{-7}$ & $0$\\
\midrule
ISNR(300) & $2.081946$ & $3.101028$ & $3.492989$ & $3.499428$ & $3.511135$ &  $3.511134$\\
\bottomrule
\end{tabular}
\caption{The ISNR values after 300 iterations for different choices of $\beta$.}
\end{table}

\begin{figure}[H]	
	\centering
	\captionsetup[subfigure]{position=top}
	{\includegraphics*[viewport= 112 224 526 601, width=0.8\textwidth]{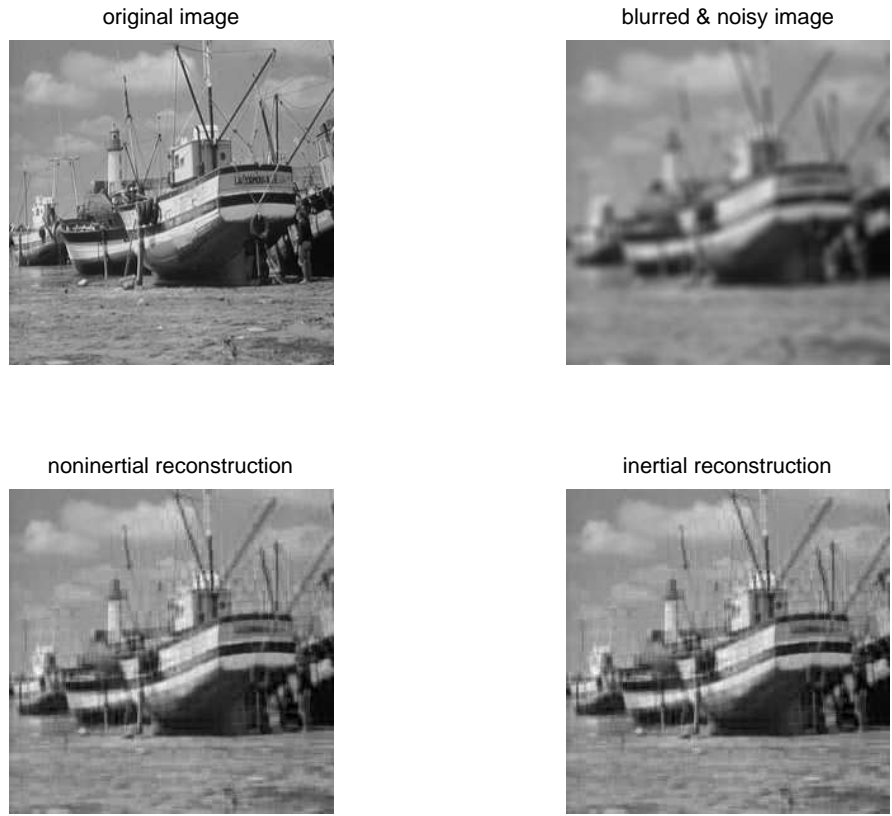}}
	\caption{\small The first row shows the original $256\times 256$ boat test image and the blurred and noisy one and the second row the reconstructed images after 300 iterations.}
	\label{fig:boat}	
\end{figure}


\begin{thebibliography}{99}
\bibitem{alvarez2000} F. Alvarez, {\it On the minimizing property of a second order dissipative system in Hilbert spaces}, SIAM Journal
on Control and Optimization 38(4), 1102--1119, 2000

\bibitem{alvarez2004} F. Alvarez, {\it Weak convergence of a relaxed and inertial hybrid projection-proximal point algorithm for
maximal monotone operators in Hilbert space}, SIAM Journal on Optimization 14(3), 773--782, 2004

\bibitem{alvarez-attouch2001} F. Alvarez, H. Attouch, {\it An inertial proximal method for maximal monotone operators via discretization
of a nonlinear oscillator with damping}, Set-Valued Analysis 9, 3--11, 2001

\bibitem{attouch-bolte2009} H. Attouch, J. Bolte, {\it On the convergence of the proximal algorithm for nonsmooth functions involving analytic
features}, Mathematical Programming 116(1-2) Series B, 5--16, 2009

\bibitem{att-b-red-soub2010} H. Attouch, J. Bolte, P. Redont, A. Soubeyran, {\it Proximal alternating minimization and projection
methods for nonconvex problems: an approach based on the Kurdyka-\L{}ojasiewicz inequality}, Mathematics of Operations Research 
35(2), 438--457, 2010

\bibitem{att-b-sv2013} H. Attouch, J. Bolte, B.F. Svaiter, {\it Convergence of descent methods for semi-algebraic and tame problems: 
proximal algorithms, forward-backward splitting, and regularized Gauss-Seidel methods}, Mathematical Programming 137(1-2) Series A, 91--129, 2013

\bibitem{att-peyp-red} H. Attouch, J. Peypouquet, P. Redont, {\it A dynamical approach to an inertial forward-backward algorithm 
for convex minimization}, SIAM Journal on Optimization 24(1), 232--256, 2014

\bibitem{bauschke-book} H.H. Bauschke P.L. Combettes, {\it Convex Analysis and Monotone Operator Theory in Hilbert Spaces}, CMS Books in Mathematics, Springer, New York, 2011

\bibitem{BecTeb09}
A. Beck and M. Teboulle, {\it A fast iterative shrinkage-thresholding algorithm for linear inverse problems}, SIAM Journal of Imaging Sciences 2(1), 183--202, 2009

\bibitem{bertsekas} D.P. Bertsekas, {\it Nonlinear Programming}, 2nd ed., Athena Scientific, Cambridge, MA, 1999

\bibitem{b-d-l2006} J. Bolte, A. Daniilidis, A. Lewis, {\it The \L{}ojasiewicz inequality for nonsmooth subanalytic functions with applications 
to subgradient dynamical systems}, SIAM Journal on Optimization 17(4), 1205--1223, 2006

\bibitem{b-d-l-s2007} J. Bolte, A. Daniilidis, A. Lewis, M. Shota, {\it Clarke subgradients of stratifiable functions}, 
SIAM Journal on Optimization 18(2), 556--572, 2007

\bibitem{b-d-l-m2010} J. Bolte, A. Daniilidis, O. Ley, L. Mazet, {\it Characterizations of \L{}ojasiewicz inequalities:
subgradient flows, talweg, convexity}, Transactions of the American Mathematical Society 362(6), 3319--3363, 2010

\bibitem{b-sab-teb} J. Bolte, S. Sabach, M. Teboulle, {\it Proximal alternating linearized minimization 
for nonconvex and nonsmooth problems}, Mathematical Programming Series A (146)(1--2), 459--494, 2014

\bibitem{b-c-inertial} R.I. Bo\c t, E.R. Csetnek, {\it An inertial forward-backward-forward primal-dual splitting algorithm for solving monotone 
inclusion problems}, arXiv:1402.5291, 2014

\bibitem{b-c-inertial-admm} R.I. Bo\c t, E.R. Csetnek, {\it An inertial alternating direction method of multipliers}, to appear in 
Minimax Theory and its Applications, arXiv:1404.4582, 2014

\bibitem{b-c-inertialhybrid} R.I. Bo\c t, E.R. Csetnek, {\it A hybrid proximal-extragradient algorithm with inertial
effects}, arXiv:1407.0214, 2014

\bibitem{b-c-inertial-nonc-ts} R.I. Bo\c t, E.R. Csetnek, 
{\it An inertial Tseng's type proximal algorithm for nonsmooth and nonconvex optimization problems}, arXiv:07241406.0724, 2014

\bibitem{b-c-h-inertial} R.I. Bo\c t, E.R. Csetnek, C. Hendrich, {\it Inertial Douglas-Rachford splitting for monotone inclusion problems}, 
arXiv:1403.3330v2, 2014

\bibitem{cabot-frankel2011} A. Cabot, P. Frankel, {\it Asymptotics for some proximal-like method involving inertia and
memory aspects}, Set-Valued and Variational Analysis 19, 59--74, 2011

\bibitem{chan-ma-yang} R.H. Chan, S. MA, J. Yang, {\it Inertial primal-dual algorithms for structured convex optimization},  
arXiv:1409.2992v1, 2014

\bibitem{chen-ma-yang} C. Chen, S. MA, J. Yang, {\it A general inertial proximal point method for mixed variational inequality problem}, 
arXiv:1407.8238v2, 2014

\bibitem{c-pesquet-r} E. Chouzenoux, J.-C. Pesquet, A. Repetti, {\it Variable metric forward-backward algorithm for minimizing the sum of a 
differentiable function and a convex function}, Journal of Optimization Theory and its Applications 162(1), 107--132, 2014

\bibitem{combettes} P.L. Combettes, {\it Solving monotone inclusions via compositions of nonexpansive averaged operators}, Optimization 53(5-6), 475--504, 2004

\bibitem{f-g-peyp} P. Frankel, G. Garrigos, J. Peypouquet, {\it Splitting methods with variable metric for Kurdyka-\L{}ojasiewicz functions and general 
convergence rates}, Journal of Optimization Theory and its Applications, DOI 10.1007/s10957-014-0642-3

\bibitem{h-l-s-t} R. Hesse, D.R. Luke, S. Sabach, M.K. Tam, {\it Proximal heterogeneous block input-output method and application 
to blind ptychographic diffraction imaging}, arXiv:1408.1887v1, 2014

\bibitem{kurdyka1998} K. Kurdyka, {\it On gradients of functions definable in o-minimal structures}, 
Annales de l'institut Fourier (Grenoble) 48(3), 769--783, 1998

\bibitem{lojasiewicz1963} S. \L{}ojasiewicz, {\it Une propri\'{e}t\'{e} topologique des sous-ensembles analytiques r\'{e}els}, 
Les \'{E}quations aux D\'{e}riv\'{e}es Partielles, \'{E}ditions du Centre National de la Recherche Scientifique Paris, 87--89, 1963 

\bibitem{mainge2008} P.-E. Maing\'{e}, {\it Convergence theorems for inertial KM-type algorithms}, Journal of Computational
and Applied Mathematics 219, 223--236, 2008

\bibitem{mainge-moudafi2008} P.-E. Maing\'{e}, A. Moudafi, {\it Convergence of new inertial proximal methods for dc programming},
SIAM Journal on Optimization 19(1), 397--413, 2008

\bibitem{boris-carte} B. Mordukhovich, {\it Variational Analysis and Generalized Differentiation, I: Basic Theory, II: Applications}, 
Springer-Verlag, Berlin, 2006.

\bibitem{moudafi-oliny2003} A. Moudafi, M. Oliny, {\it Convergence of a splitting inertial proximal method for monotone
operators}, Journal of Computational and Applied Mathematics 155, 447--454, 2003

\bibitem{nes} Y. Nesterov, {\it Introductory Lectures on Convex Optimization: A Basic Course}, Kluwer Academic Publishers, Dordrecht, 2004

\bibitem{ipiano} P. Ochs, Y. Chen, T. Brox, T. Pock, {\it iPiano: Inertial proximal algorithm for non-convex optimization}, SIAM Journal on 
Imaging Sciences 7(2), 1388--1419, 2014

\bibitem{pesq-pust} J.-C. Pesquet, N. Pustelnik, {\it A parallel inertial proximal optimization method},  
Pacific Journal of Optimization 8(2), 273--306, 2012

\bibitem{rock-wets} R.T. Rockafellar, R.J.-B. Wets, {\it Variational Analysis}, Fundamental Principles of Mathematical Sciences 317, 
Springer-Verlag, Berlin, 1998


\end{thebibliography}
\end{document}